\documentclass{article}
\usepackage[utf8]{inputenc}
\usepackage{amsfonts,amssymb,amsmath,amsthm}
\usepackage[a4paper, total={6in,9in}]{geometry}

\usepackage[utf8]{inputenc}
\usepackage{graphicx}
\usepackage{mathtools}

\usepackage{subcaption}
\usepackage{tikz}
\usepackage{enumerate}
\usepackage{cases}
\usepackage{multicol}
\usepackage{sidecap}
\usepackage{float}
\usepackage{xcolor}
\definecolor{blue2}{rgb}{0.67, 0.9, 0.93}
\usepackage[color=blue2]{todonotes}
\usepackage{cases}
\usepackage{extarrows}
\usepackage{soul}
\usepackage[pagewise]{lineno}%\linenumbers

\colorlet{siaminlinkcolor}{green!50!black}
\colorlet{siamexlinkcolor}{red!50!black}
\colorlet{siamreviewcolor}{black!50}

\usepackage[normalem]{ulem}
\usepackage[colorlinks=true,hyperfootnotes=false]{hyperref}
\hypersetup{
  colorlinks = true,
  allcolors = siaminlinkcolor,
  urlcolor = siamexlinkcolor,
}
\usepackage[nameinlink]{cleveref}
\usepackage{setspace}
\numberwithin{equation}{section}
\newtheorem{theorem}{Theorem}[section]
\newtheorem{lemma}[theorem]{Lemma}
\newtheorem{proposition}[theorem]{Proposition}
\newtheorem{corollary}[theorem]{Corollary}
\newtheorem*{remark}{Remarks}

\newcommand{\R}{{\mathbb R}}

\renewcommand{\d}{{\mathrm d}}
\DeclareMathOperator{\sech}{sech}
\usepackage{xspace}

\usepackage{authblk}
\title{\sc Soliton Amplification in the Korteweg-de Vries Equation by Multiplicative Forcing}
\author[1]{R.W.S. Westdorp \thanks{\tt \href{mailto:r.w.s.westdorp@math.leidenuniv.nl}{r.w.s.westdorp@math.leidenuniv.nl}}}
\author[1]{H.J. Hupkes \thanks{\tt \href{mailto:hhupkes@math.leidenuniv.nl}{hhupkes@math.leidenuniv.nl}}}
\affil[1]{\small Mathematisch Instituut, Universiteit Leiden, P.O. Box 9512, 2300 RA Leiden, The Netherlands}

\begin{document}

\maketitle
%\todo[color=yellow]{Consider \textbf{Diffusion-driven} propagation reversal...}
\begin{abstract}
We study the stability and dynamics of solitons  in the Korteweg-de Vries (KdV) equation with small multiplicative forcing. Forcing breaks the conservative structure of the KdV equation, leading to substantial changes in energy over long times. We show that, for small forcing, the inserted energy is almost fully absorbed by the soliton, resulting in a drastically changed amplitude and velocity. We decompose the solution to the forced equation into a modulated soliton and an infinite dimensional perturbation. Assuming slow exponential decay of the forcing, we show that the perturbation decays at the same exponential rate in a weighted Sobolev norm centered around the soliton.
\end{abstract}



%%%%%%%%%%%%%%%%%%%%%%%%%%%%%%%%%%%%%%%%%%%%%%%%%%%%%%
%                   6. BODY
%%%%%%%%%%%%%%%%%%%%%%%%%%%%%%%%%%%%%%%%%%%%%%%%%%%%%%

% Only the first word and proper nouns of section titles should be capitalized.
% The title of section 1:
\section{Introduction}

We study the forced Korteweg-de Vries equation 
\begin{align}u_t=-\partial_x^3u-2u\partial_x u+\epsilon f(\epsilon t/E) u,\label{eqn:forcedkdv}\end{align}
where $u$ is a real-valued function on $(t,x)\in(\R^+,\R)$ and $f$ is an integrable time-dependent forcing term. The small parameter $\epsilon>0$ controls the amplitude of the forcing, while the (potentially large) parameter $E > 0$ is a measure for the total supplied energy. Our main goal is to understand the effect of this forcing on the family of soliton solutions to the unforced system.

Forced KdV equations such as \eqref{eqn:forcedkdv} appear in the study of wave-phenomena subject to external disturbing mechanisms. 
Motivated by physical considerations (such as pressure
inhomogeneities or bottom topographies),
various types of forcing have been considered; see for instance
%Various types of forcing arise from a physical application (i.e. pressure inhomogeneity or bottom topography), see for instance 
\cite{pelinovsky, physics1, physics2, physics3, physics4, physics5, physics6}. The multiplicative form of the forcing term in \eqref{eqn:forcedkdv} can be thought of 
as a generic mechanism to modify the amount of energy
present in the system.
%insert energy into the system.
%
%as a model for a
%modelling a 
%generic energy inserting mechanism. 
For our purposes, \eqref{eqn:forcedkdv} constitutes a toy model that facilitates the rigorous study of perturbed waved phenomena. In particular, we view this work as a step towards establishing rigorous long-time stability of the KdV solitary waves under stochastic forcing, extending the preliminary results %as studied 
in \cite{westdorp}.

In the absence of forcing ($f(t)\equiv 0$), \eqref{eqn:forcedkdv} is the well-known KdV equation: a well-studied dispersive PDE that first appeared in the description of shallow water waves in a longitudinal canal \cite{korteweg}. Among its most notable features is the existence of soliton solutions $u(t,x)=\phi_c(x-ct)$
of the form
\begin{align}\phi_c(x) = \tfrac{3c}{2} \sech^2(\sqrt{c}x/2), \quad c > 0,\label{eqn:soliton}\end{align}
which mark a balance between dispersive and nonlinear effects. As seen in \eqref{eqn:soliton}, the solitary waves $\phi_c$ satisfy the self-similarity property 
$\phi_c(x)=c\phi_1(\sqrt{c}x)$, owing to the scaling invariance
\begin{align}
    u(t,x)\mapsto \alpha^2 u(\alpha^3 t,\alpha x)\label{eqn:symmetry}
\end{align}
of the KdV equation.

Another celebrated quality of the KdV equation is that it is completely integrable, which means that it enjoys an infinite amount of conserved quantities. In particular, the KdV flow conserves the $L^2$-norm  
\[\mathcal{N}[u]=\int_{\R}u^2\d x\]
and the Hamiltonian
\[\mathcal{H}[u]=\int_{\R}\tfrac{1}{2}(\partial_x u)^2-\tfrac{1}{3}u^3 \d x.\]
Introducing the forcing term in \eqref{eqn:forcedkdv} breaks this conservative structure. The $L^2$-norm, for instance, evolves as
\begin{align}\mathcal{N}[u(t)]=\mathcal{N}[u(0)]e^{2E\int_0^{\epsilon t/E}f(s)\d s}.\label{eqn:energyevo}\end{align}
With $\mathcal{N}$, $\mathcal{H}$, and other KdV-invariants undergoing slow (but eventually large) changes, we may expect significant consequences for the propagation of the solitons \eqref{eqn:soliton}. Using \eqref{eqn:energyevo} and the relation $\mathcal{N}[\phi_c]=6c^{3/2}$, we can heuristically predict that the amplitude of a soliton starting at $c(0)>0$ will approximately evolve according to
%\[c^{3/2}(t)\approx c_*^{3/2}e^{2E\int_0^{\epsilon t/E} f(s)\d s}. \]
\begin{align}c_{\mathrm{ap}}(t)=c(0)e^{\frac{4}{3}E\int_0^{\epsilon t/E} f(s) \d s}.\label{eqn:leadingorderc}\end{align}
We will show that this description is valid to leading-order in the small parameter $\epsilon$. If, for instance, $f(t)=e^{- t}$ and $E = \tfrac{3}{4}\ln 2$, then the soliton amplitude roughly doubles in size over time. Letting $c(t)$ denote the evolution of the soliton amplitude over time, we furthermore derive that the soliton phase, starting at a position $\xi(0)\in \R$, evolves according to
 \begin{align} \xi_{\mathrm{ap}}(t)=\xi(0)+\int_0^t c(s) \d s+\tfrac{2}{3}\epsilon \int_0^t \frac{f(\epsilon s/E)}{c^{1/2}(s)} \d s,\label{eqn:leadingorderxi}\end{align}
to leading-order in $\epsilon$.

In this work, we establish the orbital stability of the traveling-wave family \eqref{eqn:soliton} under the influence of multiplicative forcing. We show that solitons evolving via \eqref{eqn:forcedkdv} remain close to the family \eqref{eqn:soliton} in the $H^1$-norm, while undergoing a potentially large change in amplitude. In particular, we supply \eqref{eqn:forcedkdv} with the initial condition 
\begin{align}u(0,x)=\phi_{c_*}(x)+\overline{v}_*(x) \label{eqn:initial}\end{align} 
for some $c_*>0$, where $\overline{v}_*\in H^2$ and $e^{w x}\overline{v}_*\in H^1$ are suitably small  for some weight $w\in(0,\sqrt{c_*}/3)$. If the forcing term $f$ is assumed to be exponentially decaying, then the solution actually converges to a limiting wave profile in $H^1$ with an exponential weight centered around the soliton. Our main interest in pursuing this line of work is to open up rigorous stability results for solitons undergoing large amplitude changes. As such, we view this work as a step towards understanding the stability of solitons under more general perturbations, such as stochastic forcing \cite{westdorp}, as well as the stability of KdV-like quasi-solitons such as micropterons and nanopterons in systems/lattices with periodic structure \cite{faver, mcginnis, faverripples, hoffman, faverwright}. Indeed, these gradually decrease in amplitude over time when perturbed due to the presence of small oscillatory tails that interact with the perturbation.

\subsection*{Soliton stability}

Stability of the soliton family \eqref{eqn:soliton} under small initial perturbations in the KdV equation has long been established in various forms  \cite{bona, pegoweinstein, merle}. The pioneering work \cite{bona} by Bona, Souganidis and Strauss proves that the soliton family \eqref{eqn:soliton} is orbitally stable in $H^1$ via energy methods. Pego and Weinstein expand on this result in \cite{pegoweinstein} by showing that, up to a small change in the speed and the phase of the soliton, small perturbations decay when measured in the exponentially weighted spaces
\[L^2_{w}(\R)=\{g: e^{wx}g\in L^2(\R) \} \quad \text{with} \quad \|g\|_{L^2_w}=\|e^{wx}g\|_{L^2}\]
and
\begin{align}
    H^1_{w}(\R)=\{g: e^{wx}g\in H^1(\R) \} \quad \text{with} \quad \|g\|_{H^1_w}=\|e^{wx}g\|_{H^1}\label{eqn:H1a}\end{align}
for $w\in (0,\sqrt{c_*/3})$. 

Our main theorem generalizes this classic stability result to the setting of \eqref{eqn:forcedkdv}. The main new feature is that we are able to track the amplitude and phase changes introduced by the forcing, which can be of arbitrary size.

\begin{theorem}[See \Cref{sec:proofmain}]\label{thm:bigthm}
Pick $c_*,E_{\mathrm{max}}>0$, $w\in(0,\sqrt{c_*}/3)$, and $p\in[0,\tfrac{1}{4})$. There exist a weight $w_{\infty}\in(0,w)$  and constants $\delta_1,C_1,c_{\mathrm{min}},c_{\mathrm{max}}>0$ with $c_{\mathrm{min}}<c_*<c_{\mathrm{max}}$ such that the following holds true. For each $ E\in (0, E_{\mathrm{max}}]$, each $\epsilon\in (0,1]$ that satisfies
\[ \epsilon^{p}+\epsilon^{1-4p}\leq \delta_1/E, \quad \epsilon^{1-p}\leq E , \quad \epsilon\leq \delta_1 E,\]
each $\overline{v}_*\in H^2\cap H^1_w$ for which $
\|\overline{v}_*\|^2_{H^1}+\|\overline{v}_*\|^2_{H^1_w}\leq \delta_1\epsilon/E$,
and each continuous function $f:\R^+\to \R$ that satisfies the bound
\begin{align}
|f(t)|\leq & e^{- t}, \quad t\geq 0,\label{eqn:fassumption}\end{align}
there exist modulation functions $c,\xi \in C^1(\R^+;\R)$ associated to the solution $u$ of \eqref{eqn:forcedkdv} with \eqref{eqn:initial} that satisfy the following properties:
\begin{enumerate}
\item In the weighted space $H^1_{w_\infty}$, we have the exponential decay
%We have the bound 
\[\sup_{t \geq 0}e^{\epsilon t/E
}\|u(t,\cdot+\xi(t))-\phi_{c(t)}\|_{H^1_{w_\infty}}\leq C_1 \Big(\epsilon^{1-2p}+\|\overline{v}_*\|_{H^1}+\|\overline{v}_*\|_{H^1_w}\Big).\]
    \item In the unweighted space $H^1$, we have the stability bound
    %It holds that 
    \begin{align*}\sup_{t \geq 0}\|u(t,\cdot+\xi(t))-\phi_{c(t)}\|^2_{H^1}\leq & C_1 E\Big(\epsilon^{1-4p}+\|\overline{v}_*\|_{H^1}+\|\overline{v}_*\|_{H^1_w}\Big)\\
    &+C_1\frac{E}{\epsilon}\big(\|\overline{v}_*\|^2_{H^1}+\|\overline{v}_*\|^2_{H^1_w}\big).\end{align*}
    \item The amplitude function $c(t)$ takes values in $[c_{\mathrm{min}},c_{\mathrm{max}}]$, and can be approximated by
    \begin{align*}\sup_{t \geq 0}\big|c(t)\!-\!c_{\mathrm{ap}}(t)\big|
    \leq  C_1\|\overline{v}_*\|_{H^1_w}\!+\!C_1E\big(\epsilon^{1-4p}\!+\!\|\overline{v}_*\|_{H^1}\big)
    \!+\!C_1\frac{E}{\epsilon}\big(\|\overline{v}_*\|^2_{H^1}+\|\overline{v}_*\|^2_{H^1_w}\big),\end{align*}
    where $c_{\mathrm{ap}}(t)$ is defined in \eqref{eqn:leadingorderc}. Moreover, $c(t)$ converges as $t\to \infty$.
    \item The position function $\xi(t)$ satisfies \begin{align*}\sup_{t \geq 0}\left|\xi(t)\!-\!\xi_{\mathrm{ap}}(t)\right|
    \leq  C_1\|\overline{v}_*\|_{H^1_w}\!+\!C_1E\big(\epsilon^{1-4p}\!+\!\|\overline{v}_*\|_{H^1}\big)
    \!+\!C_1\frac{E}{\epsilon}\big(\|\overline{v}_*\|^2_{H^1}+\|\overline{v}_*\|^2_{H^1_w}\big),\end{align*}
    where $\xi_{\mathrm{ap}}(t)$ is defined in \eqref{eqn:leadingorderxi}.
\end{enumerate} 

\end{theorem}
% It also holds that
% \begin{align*}
% |{c}(t)-c_{\infty}|\lesssim & \epsilon^{1-4p} e^{-b\epsilon t}
% \end{align*}
% with 
% \[\tfrac{c_\infty}{c_0}=e^{\frac{4}{3}\int_0^\infty f(s)\d s}+O(\epsilon^{1-4p}).\] 
\begin{remark}
\begin{enumerate}
    \item The classic result in \cite{pegoweinstein} can be retrieved by letting $\epsilon=\delta_1 E \to 0$, up to a small loss in the weight 
    (since  $w_\infty < w$), which is further discussed in \Cref{sec:longtime}. In this case, it is required that $p=0$ through the assumption $ \epsilon^p\leq \delta_1/E$. The case of large amplitude modulation $(E \gg 0)$ requires that $p>0$.
    \item In case we replace the assumption  on $\overline{v}_*$ by the stronger assumption $
\|\overline{v}_*\|_{H^1}+\|\overline{v}_*\|_{H^1_w}\leq \delta_1\epsilon/E$, the bounds in items 2-4 simplify to
\begin{align*}\sup_{t \geq 0}\|u(t,\cdot+\xi(t))-\phi_{c(t)}\|^2_{H^1}\leq & C_1 E\epsilon^{1-4p}+C_1(E+\delta_1)\Big(\|\overline{v}_*\|_{H^1}+\|\overline{v}_*\|_{H^1_w}\Big),\end{align*}
    and
    \begin{align*}\sup_{t \geq 0}\big|c(t)-c_{\mathrm{ap}}(t)\big|+\sup_{t \geq 0}\left|\xi(t)-\xi_{\mathrm{ap}}(t)\right|
    \leq & 2C_1E\epsilon^{1-4p}+3C_1\|\overline{v}_*\|_{H^1_w}\\
    &+2C_1(E+\delta_1)\|\overline{v}_*\|_{H^1}.\end{align*}
\item Property \eqref{eqn:fassumption} is needed to obtain exponential decay of $\|v(t)\|_{H^1_{w_\infty}}$. We believe that this condition can be relaxed to $f\in L^1(0,\infty)\cap L^\infty(0,\infty)$. In this case, exponential decay over time of the weighted norm can not be expected.
\item The integrability of $f$ ensures that $c(t)$ and $c^{-1}(t)$ remain bounded, which prevents technical complications. First, it allows us to construct bounds that do not depend on $c(t)$. Second, it guarantees that we can apply a minimal exponential weight on the modulated soliton $\phi_{c(t)}$.
\item  The assumption $\overline{v}_*\in H^2\cap H^1_w$ fits the well-posedness results established in \cite[Appendix A]{pegoweinstein}. It would be interesting to see if this assumption can be relaxed to $\overline{v}_*\in L^2$ by following the arguments of \cite{mizumachi}.
    % We remark that well-posedness in the unweighted spaces $H^s$ has been established for $s>-3/4$ by Bourgain \cite{bourgain}. 
    \item  The assumption $w<\tfrac{1}{3}\sqrt{c_*}$ is slightly stricter than the common assumption $w<\sqrt{c_*/3}$. We require this stricter bound at various points to establish sharper bounds than previous works.
    \end{enumerate}
\end{remark}
Two related results for different forcing types are available in  \cite{holmer, zhong}. The result in \cite{holmer} deals with finite time stability, and in \cite{zhong}, the authors use a factorization technique that is not available in the current setting.

\subsection*{Approach}

Our approach is based on the stability theory of the solitons $\phi_c$ in the exponentially weighted spaces \eqref{eqn:H1a} developed in \cite{pegoweinstein}. The exponential weight facilitates stronger stability properties of the operator 
\begin{align}\mathcal{L}_c = -\partial_{x}^3+(c-2\phi_c) \partial_{x}-2\partial_x \phi_c = -\partial_{x}^3+c\partial_x -2\partial_x(\phi_c \cdot),\label{eqn:linearop}
\end{align}
which is associated with the linearization of the KdV equation around the soliton $\phi_c$.  Pego and Weinstein established that $\mathcal{L}_c$ generates a $C_0$-semigroup $\{e^{\mathcal{L}_{c}t}\}_{t\geq 0}$ on $L^2_w$ which is exponentially stable on the subspace of functions $v\in L^2_w$ that satisfy the orthogonality conditions
\begin{align}\label{eqn:orthogonalityct}
\langle v, \zeta_c \rangle_{L^2}= \langle v, \phi_c \rangle_{L^2}=0,
\end{align} 
where $\zeta_c$ is the primitive
\begin{align*}\zeta_c (x)=\int_{-\infty}^x \partial_c \phi_c(y)\d y \in L^2_{-w};\end{align*}
see \Cref{app:old} for further details. The conditions \eqref{eqn:orthogonalityct} play a central role in our approach. 

The main obstacle in proving \Cref{thm:bigthm} using the linear stability tools developed in \cite{pegoweinstein} is that, due to the large changes in $c(t)$, it is not feasible to linearize \eqref{eqn:forcedkdv} around a soliton with fixed amplitude. Instead, we move to a co-moving frame where the solution is not only translated, but also rescaled according to the natural scaling of the soliton family \eqref{eqn:soliton}. More precisely, we introduce the remainder
\begin{align}v(t,x)=\alpha^2(t)u\bigl(t,\alpha(t) x+\xi(t)\bigr)-\phi_{c_0}(x), \label{eqn:decomposition}\end{align}
in reference to a soliton $\phi_{c_0}$ with fixed amplitude and position. The remainder $v$ then follows an evolution equation of the form
\[v_t= \alpha^{-3}\mathcal{L}_{c_0}v+\frac{\alpha_t}{\alpha}x\partial_x v+O(\epsilon+v^2),\]
which allows us to leverage the linear stability properties of $\mathcal{L}_{c_0}$ on exponentially weighted spaces.
The term $x\partial_x v$ arising from the dilation of $u$ in \eqref{eqn:decomposition} causes significant technical complications. In order to estimate $x\partial_x v$ in terms of $v$, one needs to obtain some extra control at $|x|\to \infty$. 
For $v$ supported on $[0,\infty)$, we do so by estimating 
\[\|x\partial_x v\|_{L^2_w}\leq C_\beta\|\partial_x v\|_{L^2_{w+\beta}},\]
for some small $\beta>0$, which establishes control of the troublesome term viewed as an operator between \textit{different} exponentially weighted spaces. We arrive at a consistent argument by continuously decreasing the exponential weight over time. This, however, presents another difficulty on short-time scales, since the constant $C_\beta$ blows up as $\beta\downarrow 0$. We remedy this problem by employing the classical stability argument of \cite{pegoweinstein} on short time-scales where $c(t)$ only undergoes small fluctuations.

As has been argued by Pego \& Weinstein, stability in exponentially weighted spaces still requires control of the \textit{unweighted} $H^1$-norm of the perturbation, due to the nonlinearity in \eqref{eqn:forcedkdv} which would otherwise require double the exponential weight. This is established in \cite{pegoweinstein} by exploiting the fact that $\phi_c$ is a critical point of the conserved functional 
\begin{align}\mathcal{E}_{c}[u]=\mathcal{H}[u]+\tfrac{1}{2}c\mathcal{N}[u].\end{align}
We generalize this argument to accommodate for large fluctuations in $c(t)$ by analyzing the evolution of $\mathcal{E}_{c(t)}[u(t)]$ over time. We compute that
\[\partial_t\big(\mathcal{E}_{c(t)}[u(t)]-\mathcal{E}_{c(t)}[\phi_{c(t)}]\big)=O(\epsilon v+v^2),\]
which allows us to control the $H^1$ norm of the perturbation in terms of itself and the weighted norm.  

Our combined argument lifts the restrictions on the size of $|c_*-c(t)|$ inherent to previous approaches to establish stability. This is particularly useful for studying the KdV soliton in settings where $c(t)$ naturally undergoes large fluctuations on short time-scales. Indeed, we are pursuing the techniques developed in his paper in order to achieve stability on long time-scales in the setting of stochastic multiplicative forcing \cite{westdorp}.

% The outline is not required, but we show an example here.
\subsection*{Outline}

The paper is organized as follows. In \Cref{sec:modulation}, we derive a system of modulation equations that governs the behavior of the soliton amplitude $c(t)$, the position $\xi(t)$, and remainder $v(t)$. Then, in \Cref{sec:linearstability} we introduce the function spaces that are central to our stability argument, and assert stability and smoothing properties of the operator $\mathcal{L}_c$ on these spaces. We proceed by establishing control of the remainder $v$ over short-timescales by adapting a Duhamel argument of \cite{pegoweinstein}, and introduce the notion of time-varying weights. Thereafter, in \Cref{sec:longtime} we show that the remainder $v$ can be controlled over long time-scales in a weighted norm. The evolution of unweighted norms of $v$ is then analyzed in \Cref{sec:energy}. We finally provide the proof of \Cref{thm:bigthm} in \Cref{sec:proofmain}.

\section{Modulation system}
\label{sec:modulation}

In this section, we introduce our decomposition of solutions to \eqref{eqn:forcedkdv}, which forms the basis for our arguments. For convenience and brevity, we introduce the parameter $\gamma=\epsilon/E$
and recast \eqref{eqn:forcedkdv} in the form
\begin{align}\label{eqn:gammaversion}
    u_t=-\partial_x^3u-2u\partial_x u+\epsilon f(\gamma t) u.
\end{align}
In order to track how constants depend on the system parameters and the various choices that we make, we collect the various assumptions that we make throughout the paper in a number of labeled `settings'. The first of these relates to the global parameters $(c_*,E_{\mathrm{max}},w)$ and the initial condition for \eqref{eqn:gammaversion}.
% \begin{setting}\label{hyp:initial}
%     We have $c_*>0$ together with $E_{\mathrm{max}}>0$ and $ w\in(0,\sqrt{c_*}/3)$. The initial condition for \eqref{eqn:gammaversion} satisfies
% \begin{align}u(0,x)=\phi_{c_*}(x)+\overline{v}_*(x), \quad  \overline{v}_*\in H^2\cap H^1_w,\label{eqn:initial}\end{align}
% and the forcing term $f$ lies in the space $L^\infty(0,\infty)$.
% \end{setting}
\begin{itemize}
\item[\textbf{S1}]{
  \phantomsection\hypertarget{hyp:initial}{}\textit{
  We have $c_*>0$ together with $E_{\mathrm{max}}>0$ and $ w\in(0,\sqrt{c_*}/3)$. The initial condition for \eqref{eqn:gammaversion} satisfies
\[u(0,x)=\phi_{c_*}(x)+\overline{v}_*(x), \quad  \overline{v}_*\in H^2\cap H^1_w,\]
and the forcing term $f$ is continuous and lies in the space $L^\infty(0,\infty)$.}}
\end{itemize}
We now start by making an observation regarding the regularity of solutions to \eqref{eqn:gammaversion}.
\begin{lemma}\label{lem:regularity}
Assuming \hyperlink{hyp:initial}{S1} and letting $\epsilon, \gamma > 0$, the solution $u$ to \eqref{eqn:gammaversion} has regularity
\begin{align}u &\in C([0, T],H^2) \cap C^1([0, T],H^{-1}),\label{eqn:regularity}\\
e^{wx}u &\in C([0, T],H^1) \cap C^1([0, T],H^{-3}), \label{eqn:weightedregularity}\end{align}
for any $T > 0$. 
\end{lemma}
\begin{proof}
This result is established in \cite[Appendix A]{pegoweinstein} for $f(t)\equiv 0$ by modifying a well-posedness result of Kato \cite{kato}. 
To see that the arguments for local well-posedness (i.e. small $T>0$)
remain valid
%
%The regularity result remains valid 
upon including the forcing term $\epsilon f(\gamma t)$, 
%since 
we note that 
\eqref{eqn:gammaversion} is equivalent to a time-dependent KdV 
equation. Indeed, $u(t)$ solves \eqref{eqn:gammaversion} if and only if $z(t)=e^{-(\epsilon/\gamma)\int_0^{\gamma t} f(s) \d s}u(t)$ solves 
\begin{align*}
 z_t=-\partial_x^3z-2e^{(\epsilon/\gamma)\int_0^{\gamma t} f(s) \d s}z\partial_x z.
 \end{align*}

The arguments for global well-posedness (i.e. arbitrary $T >0$) rely on an a priori bound for the $H^2$-norm. We show here that 
%an a priori bound 
such a bound remains available upon including the forcing term $\epsilon f(\gamma t)$. Indeed, writing $u(t)=u$ for some $t\geq 0$ and using \eqref{eqn:gammaversion}, we compute that 
\begin{align*}
    \partial_t \|u\|^2_{L^2} \stackrel{\eqref{eqn:gammaversion}}{=}& -2\langle u, \partial_x^3 u+2 u \partial_x u\rangle +2\epsilon f(\gamma t) \langle u,u\rangle  = 2\epsilon f(\gamma t) \|u\|^2_{L^2},
\end{align*}
leading to the identity
\[\|u(t)\|^2_{L^2}=e^{2(\epsilon/\gamma)\int_0^{\gamma t} f(s) \d s}\|u(0)\|^2_{L^2}.\]
Moving on to the first derivative, an application of the Gagliardo-Nirenberg inequality yields
\[\|\partial_x u\|_{L^2}^2=2\mathcal{H}[u]+\tfrac{2}{3}\int_{\R}u^3 \d x\leq 2\Big|\mathcal{H}[u]\Big|+C\|u\|_{L^2}^{5/2}\|\partial_x u\|_{L^2}^{1/2}.\]
This inequality is of the form $x^4\leq A + B |x|$ with $A,B\geq 0$, from which we may conclude  $x^4\leq 2(A+B^{4/3})$, i.e.
\begin{align}\|\partial_x u\|_{L^2}^2 \leq 4\Big|\mathcal{H}[u]\Big|+2C^{4/3}\|u\|_{L^2}^{10/3}.\label{eqn:quartic}\end{align}
We then compute
\begin{align*}
    \partial_t \mathcal{H}[u]= \ &\partial_t\int_{\R}\tfrac{1}{2}(u_x)^2-\tfrac{1}{3}u^3 \d x=\int_{\R}u_x u_{xt}-u^2u_t \d x=-\int_{\R}(u_{xx} +u^2)u_t \d x\\
    \stackrel{\eqref{eqn:gammaversion}}{=}&\big\langle u_{xx} +u^2,u_{xxx}+(u^2)_x-\epsilon f(\gamma t)  u\big\rangle =\epsilon f(\gamma t) \|\partial_x u\|^2 -\epsilon f(\gamma t)  \int_{\R} u^3 \d x\\
    = \ & \epsilon f(\gamma t) (3\mathcal{H}[u]-\tfrac{1}{2}\|\partial_x u\|_{L^2}^2),
\end{align*}
which leads to 
\[\Big|\partial_t \mathcal{H}[u]\Big|\leq  5\epsilon |f(\gamma t)|\Big|\mathcal{H}[u]\Big|+\epsilon |f(\gamma t)|C^{4/3}\|u\|_{L^2}^{10/3}\]
and, via an application of Gr\"onwall's inequality, to the a priori bound
\[\Big|\mathcal{H}[u(t)]\Big|\leq \Big(\mathcal{H}[u(0)]+\frac{\epsilon}{\gamma} C^{4/3}\int_0^{\gamma t}|f( s)|\d s\sup_{0\leq s \leq t}\|u(s)\|_{L^2}^{10/3}\Big)e^{5\epsilon/\gamma  \int_0^{\gamma t} |f( s)| \d s}.\]
Via \eqref{eqn:quartic}, this provides an a priori bound on $\|\partial_x u(t)\|_{L^2}^2$.
Using the fact that the integral \[\mathcal{E}_2[u]=\int_\R (\partial_x^2 u)^2-\tfrac{10}{3} u (\partial_x u)^2 +\tfrac{5}{9}u^4 \d x\]
is conserved for the unperturbed KdV flow, we may use the bound $\|u\|_{L^4}^4\leq 4\|u\|_{H^1}^4$ to estimate
\begin{align*} \|\partial_x^2u\|_{L^2}^2=\mathcal{E}_2[u]+\tfrac{10}{3}\int_{\R} u(\partial_x u)^2 \d x-\tfrac{5}{9}\int_{\R}u^4 \d x\leq \Big|\mathcal{E}_2[u]\Big|+\tfrac{10}{3}\|u\|_{H^1}^3+\tfrac{20}{9}\|u\|_{H^1}^4.
\end{align*}
One may furthermore verify that
\begin{align*}
    \partial_t\mathcal{E}_2[u]
    =&2\epsilon f(\gamma t)\mathcal{E}_2[u]+\epsilon f(\gamma t)  \int_{\R} -\tfrac{10}{3} u(\partial_x u)^2+\tfrac{10}{9}  u^4 \ \d x,
\end{align*}
so that
\begin{align*}
    \Big|\partial_t\mathcal{E}_2[u]\Big|
    \leq &2\epsilon |f(\gamma t)|\Big|\mathcal{E}_2[u]\Big|+\epsilon |f(\gamma t)|\Big(  \tfrac{10}{3}\|u\|_{H^1}^3+\tfrac{10}{9}\|u\|_{H^1}^4\Big),
\end{align*}
through which one arrives at an a priori bound on $\mathcal{E}_2[u(t)]$, and hence $\|u(t)\|_{H^2}$.
\end{proof}
With these preliminaries in place, let us introduce our decomposition of solutions to \eqref{eqn:gammaversion}, which is based on \Cref{lem:implicit}. Provided that $\|\overline{v}_*\|_{L^2_w}$ is small enough, there exist unique parameters $\xi_0\in \R$ and $c_0>0$ that allow for the \textit{orthogonal} decomposition
\[u(0,x+\xi_0)=\phi_{c_0}(x)+\overline{v}_0(x) \quad \text{with} \quad \langle\overline{v}_0,\phi_{c_0}\rangle =\langle\overline{v}_0,\zeta_{c_0}\rangle=0,\]
where $\overline{v}_0\in H^2\cap H^1_w$. From there, we decompose the solution $u(t,x)$ to \eqref{eqn:gammaversion} for $t\geq0$ via
\begin{align}\overline{v}(t,x)=u(t,x+\xi(t))-\phi_{c(t)}(x),\label{eqn:unscaled}\end{align}
where the perturbation $\overline{v}$ satisfies
\begin{align}\langle \overline{v}(t,\cdot),\phi_{c(t)}\rangle=\langle \overline{v}(t,\cdot),\zeta_{c(t)}\rangle=0,\label{eqn:orthogonality}\end{align}
and $\xi, c$ are time-dependent modulation parameters. The existence, uniqueness, and continuous time-dependence of this decomposition is guaranteed by \Cref{lem:implicit} as long as $\|\overline{v}(t)\|_{L^2_w}$ is kept below some constant $\delta_2>0$. Based on this decomposition, we introduce a phase-shift parameter $\Omega$ through
\[{\xi}(t)=\xi_0+\int_0^t {c}(s)\ \d s+{\Omega}(t).\]
Lastly, we introduce a scaling parameter $\alpha$ and a \textit{rescaled} perturbation $v$ through
\begin{align}v(t,x)=\alpha^2(t)u\bigl(t,\alpha(t) x+\xi(t)\bigr)-\phi_{c_0}(x)\quad \text{with} \quad c(t)=c_0\alpha^{-2}(t),\label{eqn:valphadef}\end{align}
in which $u$ is rescaled in accordance with the scaling symmetry \eqref{eqn:symmetry}. Below, we collect various properties of $v$ and $\overline{v}$,
including the relation between their (distinct!) weighted norms.
\begin{lemma}\label{lem:scaling}
    Assuming \hyperlink{hyp:initial}{S1}, let $\overline{v}(t)\in H^2\cap H^1_b$ for some $t\geq 0$ and $b>0$. Furthermore, let $v(t)$ and $\alpha(t)$ be defined through \eqref{eqn:valphadef}. It then holds that 
    \begin{enumerate}
        \item $v(t,x)=\alpha^2(t)\overline{v}(t,\alpha(t)x)$;
        \item $\langle v(t,\cdot),\phi_{c_0}\rangle=\langle v(t,\cdot),\zeta_{c_0}\rangle=0$;
        \item ${v}(t)\in  H^2\cap H^1_{\alpha(t)b}$ with \[\|v(t)\|_{L^2_{\alpha(t)b}}=\alpha^{3/2}(t)\|\overline{v}(t)\|_{L^2_b}\quad \text{and} \quad  \|\partial_xv(t)\|_{L^2_{\alpha(t)b}}=\alpha^{5/2}(t)\|\partial_x\overline{v}(t)\|_{L^2_b}.\]
        
    \end{enumerate}
%     , and we have the identities 
%     \begin{align}\label{eqn:scaling}
% v(t,x)=\alpha^2(t)\overline{v}(t,\alpha(t)x)
% \end{align}
% and
% \begin{align}
%     \langle v(t,\cdot),\phi_{c_0}\rangle=\langle v(t,\cdot),\zeta_{c_0}\rangle=0\label{eqn:rescaled} .
% \end{align}
\end{lemma}
\begin{proof}
Item 1 follows from \eqref{eqn:unscaled} by substituting $y=\alpha(t)x=\sqrt{c_0/c(t)}x$. In the same way, item 2 follows from \eqref{eqn:orthogonality}. Finally, we compute the norms
\begin{align*}
    \|v(t)\|_{L^2_{\alpha(t)b}}^2=&\int_\R v^2(t,x)e^{2\alpha(t)bx}\d x=\alpha^4(t)\int_\R \overline{v}^2(t,\alpha(t)x)e^{2\alpha(t)bx}\d x\\
    =&\alpha^3(t)\int_\R \overline{v}^2(t,y)e^{2by}\d x=\alpha^3(t)\|\overline{v}(t)\|_{L^2_b}^2
\end{align*}
and
\begin{align*}
    \|\partial_xv(t)\|_{L^2_{\alpha(t)b}}^2=&\int_\R (\partial_x v)^2(t,x)e^{2\alpha(t)bx}\d x=\alpha^6(t)\int_\R (\partial_x\overline{v})^2(t,\alpha(t)x)e^{2\alpha(t)bx}\d x\\
    =&\alpha^5(t)\int_\R (\partial_x\overline{v})^2(t,y)e^{2by}\d x=\alpha^5(t)\|\partial_x\overline{v}(t)\|_{L^2_b}^2,
\end{align*}
which yield item 3.
\end{proof}

\subsection*{Modulation equations}
Below, we derive a system of evolution equations that governs the behavior of the modulation parameters $v(t),\alpha(t)$ and $\Omega(t)$. An application of the chain rule to \eqref{eqn:valphadef} yields
\begin{align}\label{eqn:vmod}
    v_t=\alpha^{-3}\big(\mathcal{L}_{c_0}v+N(v)\big)+R(t,v;\alpha,\Omega),
\end{align}
where $N$ is the KdV nonlinearity $N(v)=-2v\partial_x v$ and
\begin{align}\label{eqn:Rdef}
    R(t,v;\alpha,\Omega)=\frac{\alpha_t}{\alpha}(2+x\partial_x)(\phi_{c_0}+v)+\frac{\Omega_t}{\alpha}\partial_x(\phi_{c_0}+v)+\epsilon f(\gamma t) (\phi_{c_0}+v).
\end{align}
We claim that the modulation parameters $\alpha$ and $\Omega$ follow the system of equations
\begin{align}\label{eqn:modulation}
    \begin{bmatrix}
        {\alpha}_t\\
        {\Omega}_t
    \end{bmatrix}=-{\alpha}\epsilon{f}(\gamma t)K^{-1}({v})\begin{bmatrix}
        \langle \phi_{c_0}+{v},\phi_{c_0}\rangle\\
        \langle \phi_{c_0}+{v},\zeta_{c_0}\rangle 
    \end{bmatrix}-{\alpha}^{-2}K^{-1}({v})\begin{bmatrix}
        \langle N({v}),\phi_{c_0}\rangle\\
        \langle N({v}),\zeta_{c_0}\rangle 
    \end{bmatrix}
\end{align}
% \begin{align}
%     \begin{bmatrix}
%         \alpha_t\\
%         \tilde{\Omega}_t
%     \end{bmatrix}=-\alpha^4\tilde{f}K^{-1}(\tilde{v})\begin{bmatrix}
%         \langle \phi_{c_0}+\tilde{v},\phi_{c_0}\rangle\\
%         \langle \phi_{c_0}+\tilde{v},\zeta_{c_0}\rangle 
%     \end{bmatrix}-\alphaK^{-1}(\tilde{v})\begin{bmatrix}
%         \langle N(\tilde{v}),\phi_{c_0}\rangle\\
%         \langle N(\tilde{v}),\zeta_{c_0}\rangle 
%     \end{bmatrix}
% \end{align}
where
\begin{align}
    K(v)=\begin{bmatrix}
\bigl\langle (x\partial_x+2)(\phi_{c_0}+{v}),\phi_{c_0}\bigr\rangle& \bigl\langle  \partial_x {v}, \phi_{c_0}\bigr\rangle\\
\bigl\langle (x\partial_x+2)(\phi_{c_0}+{v}),\zeta_{c_0}\bigr\rangle &\bigl\langle \partial_x (\phi_{c_0}+{v}),\zeta_{c_0}\bigr\rangle
\end{bmatrix}\label{eqn:kmatrix}
\end{align}
and \[  {\alpha}(0)=1, \quad
        {\Omega}(0)=0.\] 
Indeed, this implies that 
\begin{align}\big\langle \alpha^{-3}N(v)+R(t,v;\alpha,\Omega),\phi_{c_0}\big\rangle=\big\langle \alpha^{-3}N({v})+R(t,v;\alpha,\Omega),\zeta_{c_0}\big\rangle=0, \label{eqn:vtorthogonal}\end{align}
which is necessary to ensure that
$\langle v,\phi_{c_0}\rangle=\langle v,\zeta_{c_0}\rangle=0$, and equivalently $\langle \overline{v},\phi_{c(t)}\rangle=\langle \overline{v},\zeta_{c(t)}\rangle=0$. The matrix $K(v)$ is invertible in case $\|v\|_{L^2_w}$ is suitably small, since $K(0)$ is invertible and $v\mapsto \det K(v)$ is continuous from $L^2_w$ to $\R$. Consequently, the system \eqref{eqn:vmod}, \eqref{eqn:modulation} is well-defined as long as $\|v(t)\|_{L^2_w}$ remains suitably bounded.

Setting $v=0$ reduces \eqref{eqn:modulation} to
\begin{align*}
    \begin{bmatrix}
        {\alpha}_t\\
        {\Omega}_t
    \end{bmatrix}=-{\alpha}\epsilon{f}(\gamma t)\frac{1}{9}\begin{bmatrix}
c_0^{-3/2}& 0\\
2c_0^{-2} &-2c_0^{-1/2}
\end{bmatrix}\begin{bmatrix}
        6c_0^{3/2}\\
        9
    \end{bmatrix}=-{\alpha}\epsilon{f}(\gamma t)\begin{bmatrix}
        \frac{2}{3}\\
        -\frac{2}{3}c_0^{-1/2}
    \end{bmatrix}
\end{align*}
and gives rise to the leading-order approximations $c_{\mathrm{ap}}(t)$ and $\xi_{\mathrm{ap}}(t)$ defined in \eqref{eqn:leadingorderc} and \eqref{eqn:leadingorderxi}.

We conclude this section by noting that,
as a result of \Cref{lem:regularity}, the evolution equation \eqref{eqn:vmod} is initially well-posed in $H^{-3}_b$ on $[0,T]$ for some $T>0$ and any 
\[b\in \big(0,w\min_{t\in[0,T]}\alpha(t)\big).\] 
In particular, the term $x\partial_x [\phi_{c_0}+v]$ in \eqref{eqn:Rdef} lies in  $L_b^2$ since there exists a constant $C >0$, for which we have
\begin{align*}\|x\partial_x(\phi_{c_0}+v)\|_{L_b^2}\stackrel{\eqref{eqn:valphadef}}{=}&\|x\partial_x \alpha^2u(t,\alpha\cdot +\xi)\|_{L_b^2}=\alpha^{3/2}\|x\partial_x u\|_{L_{b/\alpha}^2} \\
\leq \ & C \big((b/\alpha)^{-1}\|\partial_x u\|_{L^2}+(w-b/\alpha)^{-1}\|\partial_x u\|_{L_w^2}\big);
\end{align*}
see also \Cref{lem:observations} below.

\section{Linear stability on weighted spaces}\label{sec:linearstability}

The orbital stability proof of \cite{pegoweinstein} relies on stability and smoothing properties of the evolution generated by the linear operator $\mathcal{L}_c$ defined in \eqref{eqn:linearop}; see \Cref{thm:linearstability}. These properties hold on exponentially weighted spaces after applying the projection $Q_c=I-P_c$, where
$P_c$ is the spectral projection corresponding to the 0-eigenvalue of $\mathcal{L}_c$, given by 
\[P_{c} f=\langle f, \eta_c^1 \rangle \partial_x \phi_{c} +\langle f, \eta_c^2\rangle \partial_c \phi_c.\]
Here, $\eta_c^1 $ and $\eta_c^2$ are linear combinations of $\phi_c$ and $\zeta_c$ that are defined in \eqref{eqn:linearcomb}. As such, the subspace of $L^2_w$ characterized by \eqref{eqn:orthogonalityct} corresponds to $\ker P_c \subseteq L^2_w$. We review this classical result in \Cref{app:old}, where it is stated as \Cref{thm:linearstability}. 

The spaces $L^2_w$, however, are unsuitable for controlling the term $x\partial_x v$ present in \eqref{eqn:Rdef}. One can for instance not expect that
\[\|x\partial_x v\|_{L^2_w}\leq C \|v\|_{H^1_w}\]
for a constant $C$ and all $v\in L^2_w$, due to the unbounded and non-integrable factor $x$. It is, however, true that
\[\|x g\|^2_{L^2_w}=\int_{-\infty}^\infty e^{2wx} x^2 g^2(x) \d x \leq C\int_{-\infty}^0 e^{2b_-x} g^2(x) \d x+C\int_0^\infty e^{2b_+x} g^2(x) \d x \]
for $b_-< w<  b_+$, some constant $C>0$, and functions $g$ for which the above quantity is well-defined. This leads us to introduce the notion of \textit{asymmetrically-weighted spaces}. For every $\mathbf{w}=(w_-,w_+)\in \R^2$, we introduce the weighted space
\[L^2_{\mathbf{w}}=\big\{g: e^{w_- x}g\in L^2(-\infty,0) \quad \text{and} \quad e^{w_+ x}g\in L^2(0,\infty)\big\}\]
with norm
\begin{align}\|g\|^2_{L^2_{\mathbf{w}}}:=\int_{-\infty}^0 e^{2w_- x}g^2(x)\d x+\int_{0}^\infty e^{2w_+ x}g^2(x)\d x.\label{eqn:norm}\end{align}
Writing $g_+(x)=g(x)\chi_{x\geq 0}(x)$ and $g_-(x)=g(x)\chi_{x\leq 0}(x)$, we have
\begin{align}\|g\|^2_{L^2_{\mathbf{w}}} =\|g_-\|^2_{L^2_{w_-}}+\|g_+\|^2_{L^2_{w_+}}.\label{eqn:normrewrite}\end{align}
The following result asserts that the stability and smoothing properties of $\{e^{\mathcal{L}_ct}\}_{t\geq 0}$ provided in \Cref{thm:linearstability} extend to asymmetrically-weighted spaces.
\begin{proposition}\label{prop:smoothing}
    Let $c>0$ and $w_-,w_+\in (0,\sqrt{c})$. For all $\beta>0$ that satisfy \[\beta<\min\{w_-(c-w_-^2),w_+(c-w_+^2)\}, \] 
    there exists a constant $M>0$ such that for all $g\in L_{\mathbf{w}}^2$, $t>0$, and $k\in\{0,1\}$, we have
    \begin{align}
        \|\partial_x^k  e^{\mathcal{L}_{c}t}Q_c g\|_{L^2_{\mathbf{w}}}\leq M t^{-k/2}e^{-\beta t}\|g\|_{L^2_{\mathbf{w}}}.\label{eqn:stabilityandsmoothing}
    \end{align}
\end{proposition}
\Cref{prop:smoothing} is easily proved using some elementary observations regarding the norm \eqref{eqn:norm}. Items 2 and 3 will be used in later sections.
\begin{lemma}\label{lem:observations}
    If $\mathbf{w},\mathbf{b}\!\!\in\! \R^2$ satisfy $b_-\!<\! w_-$, $b_+\!>\! w_+$ and $g\!\in\! L^2_{\mathbf{b}}$, then $g, xg \!\in\! L^2_{\mathbf{w}}$, and
    \begin{enumerate}
        \item $\|g\|_{L^2_{\mathbf{w}}}\leq \|g\|_{L^2_{w_+}}+\|g\|_{L^2_{w_-}}$;
        \item $\|g\|_{L^2_{\mathbf{w} }}
     \leq\|g\|_{L^2_{\mathbf{b}}} $;
     \item $\|x g\|^2_{L^2_{\mathbf{w} }}
     \leq e^{-2}(b_--w_-)^{-2}\|g_-\|^2_{L^2_{b_-}}+ e^{-2}(b_+-w_+)^{-2}\|g_+\|^2_{L^2_{b_+ }}$.
    \end{enumerate}
\end{lemma}
\begin{proof}
Item 1 follows directly from \eqref{eqn:normrewrite}. For item 2, we estimate
\begin{align*}
     \|g\|^2_{L^2_{\mathbf{w} }}= \|g_-\|^2_{L^2_{w_-}}+\|g_+\|^2_{L^2_{w_+}}
     \leq &\|g_-\|^2_{L^2_{b_-}}+\|g_+\|^2_{L^2_{b_+}}=\|g\|^2_{L^2_{\mathbf{b}}}. 
\end{align*}
 Similarly, we prove item 3 by estimating
\begin{align*}
     \|x g\|^2_{L^2_{\mathbf{w} }}=& \|x g_-\|^2_{L^2_{w_-}}+\|x g_+\|^2_{L^2_{w_+}}\\
     \leq & 
     \sup_{x\leq 0}x^2e^{2(w_--b_-)x}\|g_-\|^2_{L^2_{b_-}}+ \sup_{x\geq 0}x^2e^{2(w_+-b_+)x}\|g_+\|^2_{L^2_{b_+ }}\\
     = & e^{-2}(b_--w_-)^{-2}\|g_-\|^2_{L^2_{b_-}}+ e^{-2}(b_+-w_+)^{-2}\|g_+\|^2_{L^2_{b_+ }}.     \qedhere
\end{align*}
\end{proof}
\begin{proof}[Proof of \Cref{prop:smoothing}]
For $g\in L^2_{\mathbf{w}}$, we may use item 1 of \Cref{lem:observations} and \Cref{thm:linearstability} to compute
\begin{align*}
    \|\partial_x^k e^{\mathcal{L}_{c}t}Q_c g\|_{L^2_{\mathbf{w}}}\leq \ &  \|\partial_x^k e^{\mathcal{L}_{c}t}Q_c g_-\|_{L^2_{\mathbf{w}}}+\|\partial_x^k e^{\mathcal{L}_{c}t}Q_c g_+\|_{L^2_{\mathbf{w}}}\\
    \leq \ &  \|\partial_x^k e^{\mathcal{L}_{c}t}Q_c g_-\|_{L^2_{w_-}}+\|\partial_x^k e^{\mathcal{L}_{c}t}Q_c g_-\|_{L^2_{w_+}}\\
    &+\|\partial_x^k e^{\mathcal{L}_{c}t}Q_c g_+\|_{L^2_{w_-}}+\|\partial_x^k e^{\mathcal{L}_{c}t}Q_c g_+\|_{L^2_{w_+}}\\
    \stackrel{\eqref{eqn:linear}}{\leq} & Mt^{-k/2}e^{-\beta t}(\|g_-\|_{L^2_{w_-}}+\|g_+\|_{L^2_{w_+}})\\
    \leq \ &  Mt^{-k/2}e^{-\beta t}\|w\|_{L^2_{\mathbf{w}}}.    \qedhere
\end{align*}
\end{proof}

\section{Short-time control}
\label{sec:shorttime}
In this section, we establish control over the perturbation in the original frame in asymmetrically-weighted spaces over short time-scales. Using our rescaled frame description \eqref{eqn:vmod} on \textit{short} time-scales leads to problematic complications arising from the term $x\partial_x$ in \eqref{eqn:Rdef}. Instead, we rely on classical results valid for small amplitude fluctuations. We follow the argument of Pego \& Weinstein \cite[Proposition 6.1]{pegoweinstein}, which uses the evolution
\begin{align*}
    \overline{v}_t
    =&\mathcal{L}_{c(t)}\overline{v}+N(\overline{v})-{c_t}\partial_c\phi_{c(t)}+{\Omega_t}\partial_x(\overline{v}+\phi_{c(t)})+\epsilon f(\gamma t)(\overline{v}+\phi_{c(t)}),
\end{align*}
 for the perturbation in the original frame $\overline{v}$ 
(initially justified in $H^{-1}$ via \eqref{eqn:regularity}). This argument relies heavily on an approximation of the form
\[\mathcal{L}_{c(t)}=\mathcal{L}_{c_0}+O\big(|c_0-c(t)|\big).\]

In our setting, we pick $t_\diamond > 0$ and linearize around the fixed soliton $\phi_{c(t_\diamond)}$ 
by writing
\begin{align}\overline{v}_t(t_\diamond+s)=\mathcal{L}_{c(t_\diamond)}\overline{v}(t_\diamond+s)+Y(t_\diamond,s,\overline{v};c,\Omega),\label{eqn:strong}\end{align}
where
\begin{align}
    Y(t_\diamond,s,\overline{v};c,\Omega)=&\big(\mathcal{L}_{c(t_\diamond+s)}-\mathcal{L}_{c(t_\diamond)}\big)\overline{v}(t_\diamond+s)+N(\overline{v}(t_\diamond+s))-{c_t(t_\diamond+s)}\partial_c\phi_{c(t_\diamond+s)}\nonumber\\
&+{\Omega_t(t_\diamond+s)}\partial_x \big(\overline{v}(t_\diamond+s)+\phi_{c(t_\diamond+s)}\big)\nonumber\\
&+\epsilon f(\gamma (t_\diamond+s))\big(\overline{v}(t_\diamond+s)+\phi_{c(t_\diamond+s)}\big).\label{eqn:yts}
\end{align}
We recall that $c(t)$ is related to the rescaling process through $c(t)=c_0\alpha^{-2}(t)$, so that $c(t)$ follows the modulation equation $c_t=-2c_0\alpha^{-3}\alpha_t$ where $\alpha_t$ is given by \eqref{eqn:modulation}. Throughout this section, we will assume that both $\alpha$ and $c$ can be bounded away from zero. More precisely, we make the following assumptions \hyperlink{hyp:alphabound}{S2} and \hyperlink{hyp:wmin}{S3}, and formulate a condition \hyperlink{hyp:Tproperties}{C1} that underlies most of the results in this section.
\begin{itemize}
\item[\textbf{S2}]{
  \phantomsection\hypertarget{hyp:alphabound}{}\textit{
  The constants $\alpha_{\mathrm{min}}, \alpha_{\mathrm{max}}\in \R$ satisfy $0<\alpha_{\mathrm{min}}<1 < \alpha_{\mathrm{max}}$.}}
\end{itemize}
\begin{itemize}
\item[\textbf{S3}]{
  \phantomsection\hypertarget{hyp:wmin}{}\textit{
  The constant $w_{\mathrm{min}}\in \R$ satisfies $0<w_{\mathrm{min}}<w$.}}
\end{itemize}
% \begin{setting}\label{hyp:alphabound}
% The constants $\alpha_{\mathrm{min}}, \alpha_{\mathrm{max}}\in \R$ satisfy $0<\alpha_{\mathrm{min}}<1 < \alpha_{\mathrm{max}}$.
% \end{setting}
% \begin{setting}\label{hyp:wmin}
%     The constant $w_{\mathrm{min}}\in \R$ satisfies $0<w_{\mathrm{min}}<w$.
% \end{setting}
% \begin{condition}\label{hyp:Tproperties}Given $T>0$, let $\overline{v}\in  C([0, T],H^2) \cap C^1([0, T],H^{-1})$ solve \eqref{eqn:strong} and $\alpha \in C^1([0,T],\R)$ solve \eqref{eqn:modulation}. We assume that
%     \[\alpha(t)\in [\alpha_{\mathrm{min}},\alpha_{\mathrm{max}}], \quad t\in[0,T].\]
% \end{condition}
\begin{itemize}
\item[\textbf{C1}]{
  \phantomsection\hypertarget{hyp:Tproperties}{}\textit{
  Given $T>0$, the function $\overline{v}\in  C([0, T],H^2) \cap C^1([0, T],H^{-1})$ satisfies \eqref{eqn:strong}, while the function $\alpha \in C^1([0,T],\R)$ solves \eqref{eqn:modulation}. In addition, we have the inclusion
     \[\alpha(t)\in [\alpha_{\mathrm{min}},\alpha_{\mathrm{max}}], \quad t\in[0,T].\]}}
\end{itemize}
Our main result in this section provides short-time control over $\overline{v}$. More precisely, on time intervals where $\overline{v}$ and the fluctuations of $c$ are small enough, the growth of $\overline{v}$ measured in the $H^1_{\overline{\mathbf{w}}}$-norm can be explicitly controlled by the small forcing amplitude $\epsilon > 0$ and the length of the time interval $\delta > 0$.
 
\begin{proposition}[Short-time control]\label{prop:short}
   Assuming \hyperlink{hyp:initial}{S1} - \hyperlink{hyp:wmin}{S3}, there exist constants $\delta_4,C_4>0$ such that the following holds true for each $\epsilon,\gamma>0$. If \hyperlink{hyp:Tproperties}{C1} holds for some $T>0$, then for each $t,\delta>0$ with $t+\delta\leq T$, and each ${\overline{\mathbf{w}}}=(\overline{w}_-,\overline{w}_+)\in \R^2$ with 
   \[\overline{w}_-, \overline{w}_+\in \Big[\frac{w_{\mathrm{min}}}{\alpha(t+s)},\frac{\sqrt{c_0}/3}{\alpha(t+s)}\Big], \quad s\in [0,\delta],\]
   the bound
    \begin{align}\Big(\epsilon+\sup_{s\in[0,\delta]}\big(\|\overline{v}(t+s)\|_{H^1_{\overline{\mathbf{w}}}}+
    \|\overline{v}(t+s)\|_{H^1}+
    |c(t+s)-c(t)|\big)\Big)(\sqrt{\delta}+\delta^2)\leq \delta_4\label{eqn:shorttime}\end{align} 
    implies \begin{align}\sup_{s\in[0,\delta]}\|\overline{v}(t+s)\|_{H^1_{\overline{\mathbf{w}}}}\leq C_4\Big( \|\overline{v}(t)\|_{H^1_{\overline{\mathbf{w}}}}+\epsilon \sup_{s\in[0,\delta]}|f(\gamma (t+s))| (\sqrt{\delta}+\delta^2)\Big).\label{eqn:shorttimeestimate}\end{align}
\end{proposition}

In order to apply this result to the perturbation $v$ in the rescaled frame, it is essential to note that \eqref{eqn:shorttimeestimate} transforms into an estimate between \textit{different} weighted spaces due to the time-dependent rescaling in the $x$-direction via $\alpha(t)$. To remedy this, and to deal with the problematic $x\partial_x {v}$ term in \eqref{eqn:strong2}, we introduce time-dependent weights $\mathbf{w}(t)=\big(w_-(t),w_+(t)\big)$ that increase/decrease at a rate sufficient to compensate rescaling by $\alpha$. More precisely, we assume the following.
\begin{itemize}
\item[\textbf{C2}]{
  \phantomsection\hypertarget{hyp:timevarying}{}\textit{
  Given $T\geq\delta>0$, the increasing function $w_-:\R^+\to \R^+$ and decreasing function $w_+:\R^+\to \R^+$ satisfy
\begin{align}\frac{w_-(t+s/2)}{w_-(t+s)}\leq \frac{{\alpha}(t)}{{\alpha}(t+s)}\leq \frac{w_+(t+s/2)}{w_+(t+s)}\label{eqn:weightcondition}\end{align}
for each $s\in[0,\delta]$ and $t\in [0,T-s]$. Furthermore,
\[w_{\mathrm{min}}\leq w_\pm(t) \leq w, \quad t\in[0,T].\]}}
\end{itemize}
% \begin{condition}\label{hyp:timevarying}
% Given $T\geq\delta>0$, the increasing function $w_-:\R^+\to \R^+$ and decreasing function $w_+:\R^+\to \R^+$ satisfy
% \begin{align}\frac{w_-(t+s/2)}{w_-(t+s)}\leq \frac{{\alpha}(t)}{{\alpha}(t+s)}\leq \frac{w_+(t+s/2)}{w_+(t+s)}\label{eqn:weightcondition}\end{align}
% for each $s\in[0,\delta]$ and $t\in [0,T-s]$. Furthermore,
% \[w_{\mathrm{min}}\leq w_\pm(t) \leq w, \quad t\in[0,T].\]
% \end{condition}
Assumption \hyperlink{hyp:timevarying}{C2} implies that $t\to \tfrac{w_+(t)}{{\alpha}(t)}$ is decreasing and $t\to \tfrac{w_-(t)}{{\alpha}(t)}$ is increasing on $[0,T]$, which essentially means that the weight-functions retain their monotonicity after rescaling. Since $w_-$ and $w_+$ are evaluated at $t+s/2$ in \eqref{eqn:weightcondition}, it furthermore follows that there is a lower bound on their absolute growth rate in the sense that
\[\log\big(\tfrac{w_-}{{\alpha}}\big)^\prime(t) \geq \tfrac{1}{2}\log(w_-)^\prime (t) \quad \text{and} \quad \log\big(\tfrac{w_+}{{\alpha}}\big)^\prime (t)\leq \tfrac{1}{2}\log(w_+)^\prime (t),\quad t\in[0,T].\]
With this condition in place, we formulate the following corollary to \Cref{prop:short}.
\begin{corollary}\label{cor:short}
Assuming \hyperlink{hyp:initial}{S1} - \hyperlink{hyp:wmin}{S3}, there exist constants $\delta_5, C_5>0$ so that the following holds true for each $\epsilon,\gamma>0$. If \hyperlink{hyp:Tproperties}{C1} and \hyperlink{hyp:timevarying}{C2} hold for some $T\geq \delta>0$, then for each $t\in[0,T-\delta]$, the bound
% \begin{align}\sup_{s\in [0,\delta]}\|{v}(t+s)\|_{H^1_{\mathbf{w}(t+s)}}\leq &\epsilon_1 \label{eqn:vass} \\ 
%     \sup_{s\in [0,\delta]}\|{v}(t+s)\|_{H^1}\leq &\epsilon_2 \nonumber
% \end{align}
% with 
\[\Big(\epsilon +\sup_{s\in [0,\delta]}\big(\|{v}(t+s)\|_{H^1_{\mathbf{w}(t+s)}}+\|{v}(t+s)\|_{H^1}+|c(t+s)-c(t)|\big)\Big)(\sqrt{\delta}+\delta^2)\leq \delta_5\]
implies
\[\sup_{s\in [0,\delta]}\|{v}(t+s)\|_{H^1_{\mathbf{w}(t+s)}}\leq C_5 \Big(\|{v}(t)\|_{H^1_{\mathbf{w}(t+\delta/2)}}+\epsilon \sup_{s\in[0,\delta]}|f(\gamma (t+s))| (\sqrt{\delta}+\delta^2)\Big).\]
\end{corollary}
In preparation for the proof of \Cref{prop:short}, we 
examine \eqref{eqn:modulation} and
show that $\alpha_t, \Omega_t$ can be controlled by the perturbation $v$.  In contrast to \cite{pegoweinstein}, we deal with the presence of the forcing term and require slightly sharper control on the modulation parameters.

\begin{lemma}\label{lem:estimates}
   Assuming \hyperlink{hyp:initial}{S1} and \hyperlink{hyp:alphabound}{S2}, there exist constants $\delta_6, C_6>0$ so that the following holds true for each $\epsilon,\gamma>0$. If \hyperlink{hyp:Tproperties}{C1} holds for some $T>0$, then for each $\mathbf{b}=(b_-,b_+)\in \R^2$ with $b_-, b_+\in (0,\sqrt{c_*}/3]$ and $t \in [0,T]$, the bound
\[\|v(t)\|_{L^2_{\mathbf{b}}}\leq b_+^{1/2}\delta_6\]
implies
    \begin{align}
    &|\alpha_t(t)+\tfrac{2}{3}\alpha(t) \epsilon f(\gamma t)|+|\Omega_t(t)-\tfrac{2}{3}c_0^{-1/2}\alpha(t)\epsilon f(\gamma t)| \nonumber\\&\leq  C_6 \epsilon|f(\gamma t)| b_+^{-1/2}\|v(t)\|_{L^2_{\mathbf{b}}}
    +C_6\|v(t)\|^2_{L^2_{\mathbf{b}}}, \label{eqn:modulationestimproved}
    \end{align}
and hence
  \begin{align}
        |\alpha_t(t)|+|\Omega_t(t)|\leq C_6\Big( \epsilon|f(\gamma t)| (1+b_+^{-1/2}\|v(t)\|_{L^2_{\mathbf{b}}})+\|v(t)\|^2_{L^2_{\mathbf{b}}}\Big).\label{eqn:modulationestimate}
    \end{align}
\end{lemma}

\begin{proof}
Let us write $v=v(t)$ for brevity, and rewrite \eqref{eqn:modulation} as
\begin{align}
    \begin{bmatrix}
        {\alpha}_t+\tfrac{2}{3}\alpha(t) \epsilon f(\gamma t)\\
        {\Omega}_t-\tfrac{2}{3}c_0^{-1/2}\alpha(t)\epsilon f(\gamma t)
    \end{bmatrix}
=&-{\alpha}\epsilon{f}(\gamma t)K^{-1}({v})\begin{bmatrix}
        \langle {v},\phi_{c_0}\rangle\\
        \langle {v},\zeta_{c_0}\rangle 
    \end{bmatrix}
        \nonumber
    \\
    &-\alpha f(\gamma t)(K^{-1}({v})-K^{-1}(0))\begin{bmatrix}
        \langle \phi_{c_0},\phi_{c_0}\rangle\\
        \langle \phi_{c_0},\zeta_{c_0}\rangle 
    \end{bmatrix}\nonumber\\
    &-{\alpha}^{-2}K^{-1}({v})\begin{bmatrix}
        \langle N({v}),\phi_{c_0}\rangle\\
        \langle N({v}),\zeta_{c_0}\rangle 
    \end{bmatrix} \label{eqn:modulationsharp}.\end{align}
Here we have used that
\[K^{-1}(0)\begin{bmatrix}
        \langle \phi_{c_0},\phi_{c_0}\rangle\\
        \langle \phi_{c_0},\zeta_{c_0}\rangle 
    \end{bmatrix}=\begin{bmatrix}
        \frac{2}{3}\\
        -\frac{2}{3}c_0^{-1/2}
    \end{bmatrix}.\]
Setting out to control $K^{-1}(v)$, we note that
\[K(0)=\begin{bmatrix}
9c_0^{3/2}& 0\\
9 &\tfrac{9}{2}c_0^{1/2}
\end{bmatrix}\]
is invertible, so that we can find constants $\tilde{C}_1,\tilde{C}_2>0$ that ensure $\|A^{-1}\|_{\mathrm{op}}\leq \tilde{C}_2$ for all $A\in \R^{2\times 2}$ which satisfy
\[|A_{ij}-K_{ij}(0)|\leq \tilde{C}_1, \quad  (i,j) \in \{1,2\}^2.\]
Here, $\|\cdot\|_{\mathrm{op}}$ denotes the operator-norm on $(\R^2, \|\cdot\|_1)$, chosen for convenience in the computations below. Now note that
\begin{align*}
    |K_{ij}(v)-K_{ij}(0)|\stackrel{\eqref{eqn:kmatrix}}{\leq} &\Big((2c_0+1)\|\partial_c\phi_{c_0}\|_{L^2_{-\mathbf{b}}}+\|\partial_x\phi_{c_0}\|_{L^2_{-{\mathbf{b}}}}\Big)\|v\|_{L^2_{\mathbf{b}}}\\
    &+\big(\|\zeta_{c_0}\|_{L^2_{-\mathbf{b}}}+\|x\partial_c\phi_{c_0}\|_{L^2_{-{\mathbf{b}}}}\big)\|v\|_{L^2_{\mathbf{b}}}
\end{align*}
for all $ (i,j) \in \{1,2\}^2$. Since $\partial_c\phi_{c_0}$ and $ \partial_x\phi_{c_0}$ decay exponentially as $|x|\to \infty$, we can estimate their weighted norm by a constant that does not depend on $\mathbf{b}$. The function $\zeta_{c_0}$, however, tends to $\int_{\R}\partial_c \phi_{c_0} \d x=3c_0^{-1/2}$ as $x\to \infty$, and is not an $L^2$-function. We therefore estimate
\begin{align*}
    \|\zeta_{c_0}\|^2_{L^2_{-\mathbf{b}}}\!=\!\int_{-\infty}^0 e^{-2b_- x}\zeta_{c_0}^2 (x)\d x\!+\!\int_{0}^\infty e^{-2b_+ x}\zeta_{c_0}^2 (x)\d x
    \leq \int_{-\infty}^0\zeta_{c_0}^2 (x)\d x+\frac{\|\zeta_{c_0}\|^2_{L^\infty}}{2b_+}
\end{align*}
and have
\begin{align}
    |K_{ij}(v)-K_{ij}(0)|\leq \ & \tilde{C}_3b_+^{-1/2} \|v\|_{L^2_{\mathbf{b}}}\label{eqn:entryest}
\end{align}
for all $ (i,j) \in \{1,2\}^2$ and some constant $\tilde{C}_3>0$.
% \begin{align*}\tilde{C}_3=\sup_{b_-,b_+\in[0,\sqrt{c_*/3}]}((2c_0+1)\|\partial_c\phi_{c_0}\|_{L^2_{-\mathbf{b}}}+\|\partial_x\phi_{c_0}\|_{L^2_{-{\mathbf{b}}}}+\|\zeta_{c_0}\|_{L^2_{-\mathbf{b}}}+\|x\partial_c\phi_{c_0}\|_{L^2_{-{\mathbf{b}}}}).\end{align*}
in case $\delta_6\leq b_+^{1/2}\tfrac{\tilde{C}_1}{\tilde{C}_3} $, it follows via \eqref{eqn:entryest} that $\|K^{-1}(v)\|_{\mathrm{op}}\leq \tilde{C}_2 $. Turning to the term $K^{-1}({v})-K^{-1}(0)$ in \eqref{eqn:modulationsharp}, we note that
\begin{align*}\|K^{-1}({v})-K^{-1}(0)\|_{\mathrm{op}}&=\|K^{-1}(0)\big(K(v)-K(0)\big)K^{-1}(v)\|_{\mathrm{op}}\\
&\leq\|K^{-1}(0)\|_{\mathrm{op}}\|K(v)-K(0)\|_{\mathrm{op}}\|K^{-1}(v)\|_{\mathrm{op}}\\
&\leq \tilde{C}_4 b_+^{-1/2} \|v\|_{L^2_{\mathbf{b}}}.\end{align*}
We proceed by estimating
\begin{align*}
     |\langle N(v),\phi_{c_0}\rangle|+|\langle N(v),\zeta_{c_0}\rangle|=&|\langle \partial_x(v^2),\phi_{c_0}\rangle|+|\langle \partial_x(v^2),\zeta_{c_0}\rangle|\\
     =&|\langle v^2,\partial_x\phi_{c_0}\rangle|+|\langle v^2,\partial_c\phi_{c_0}\rangle|\\
     \leq& \big\|e^{-2b_- \cdot}\chi_{x\leq 0}(|\partial_x\phi_{c_0}|+|\partial_c\phi_{c_0}|)\big\|_{L^\infty}\|v\|^2_{L^2_{\mathbf{b}}}\\
     \leq&\tilde{C}_5\|v\|^2_{L^2_{\mathbf{b}}},
\end{align*}
together with
\begin{align*}
    |\langle {v},\phi_{c_0}\rangle|+|
        \langle {v},\zeta_{c_0}\rangle|\leq& \|\phi_{c_0}\|_{L^2_{-\mathbf{b}}}\|v\|_{L^2_{\mathbf{b}}}+\tilde{C}_3b_+^{-1/2}\|v\|_{L^2_{\mathbf{b}}}.
\end{align*}
We conclude via \eqref{eqn:modulationsharp} that
\begin{align*}
    & |\alpha_t+\tfrac{2}{3}\alpha(t) \epsilon f(\gamma t)|+|\Omega_t-\tfrac{2}{3}c_0^{-1/2}\alpha(t)\epsilon f(\gamma t)|\\
    &\leq \alpha_{\mathrm{max}} \epsilon |f(\gamma t)|\tilde{C}_2 \Big(\|\phi_{c_0}\|_{L^2_{-\mathbf{b}}}\|v\|_{L^2_{\mathbf{b}}}+\tilde{C}_3b_+^{-1/2}\|v\|_{L^2_{\mathbf{b}}}\Big)\\
    &\quad +\alpha_{\mathrm{max}} \epsilon |f(\gamma t)|\tilde{C}_4  \Big(\|\phi_{c_0}\|_{L^2}^2+|\langle \phi_{c_0},\zeta_{c_0}\rangle|\Big)b_+^{-1/2}\|v\|_{L^2_{\mathbf{b}}}
     +\alpha_{\mathrm{min}}^{-2}\tilde{C}_2 \tilde{C}_5 \|v\|^2_{L^2_{\mathbf{b}}}.   \qedhere
\end{align*}
\end{proof}
Using \Cref{lem:estimates}, it is now straightforward to control the term $Y(t_\diamond,s,\overline{v};c,\Omega)$ introduced in \eqref{eqn:yts} in terms of $\overline{v}$. We note, though, that \Cref{lem:estimates} provides control over $\alpha_t$ and $\Omega_t$ in terms of a weighted norm of $v$, instead of $\overline{v}$. We remedy this using item 1 of \Cref{lem:scaling}, taking care to apply a rescaled weight. 

\begin{corollary}\label{cor:yts}
Assuming \hyperlink{hyp:initial}{S1} - \hyperlink{hyp:wmin}{S3}, there exists a constant $C_7>0$ so that the following holds true for each $\epsilon,\gamma>0$. If \hyperlink{hyp:Tproperties}{C1} holds for some $T>0$, then for any $t_\diamond,s\geq 0$ with $t_\diamond+s\in[0,T]$ and $\overline{\mathbf{w}}=(\overline{w}_-,\overline{w}_+)\in \R^2$ with 
% \[\alpha (t_\diamond+s) w_-\in[0,\sqrt{c_0}/3]\quad \text{and} \quad \alpha (t_\diamond+s) w_+\in[w_{\mathrm{min}},\sqrt{c_0}/3],\] 
\[\overline{w}_-, \overline{w}_+\in \Big[\frac{w_{\mathrm{min}}}{\alpha(t_\diamond+s)},\frac{\sqrt{c_0}/3}{\alpha(t_\diamond+s)}\Big],\]
the bound 
\[\|\overline{v}(t_\diamond+s)\|_{L^2_{\overline{\mathbf{w}}}}\leq \alpha(t_\diamond+s)^{-1}\delta_6 \overline{w}_+^{1/2}\] 
implies
    \begin{align}
        \|Y(t_\diamond,s,\overline{v};c,\Omega)\|_{L^2_{\overline{\mathbf{w}}}}\leq&C_7\Big(|c(t_\diamond+s)-c(t_\diamond)|+\epsilon |f(\gamma (t_\diamond+s))| (1+\|\overline{v}(t_\diamond+s)\|_{H^1_{\overline{\mathbf{w}}}})\nonumber\\
    &+\|\overline{v}(t_\diamond+s)\|^2_{H^1_{\overline{\mathbf{w}}}}+\|\overline{v}(t_\diamond+s)\|_{H^1}\Big)\|\overline{v}(t_\diamond+s)\|_{H^1_{\overline{\mathbf{w}}}}\nonumber\\
    &+C_6\epsilon \big|f\big(\gamma (t_\diamond+s)\big)\big|.\label{eqn:Rts}
    \end{align}
\end{corollary}
\begin{proof}
We write $\overline{v}=\overline{v}(t_\diamond+s)$ and estimate the various components in \eqref{eqn:yts}. First, we have
\begin{align*}\|(\mathcal{L}_{c(t_\diamond+s)}\!-\!\mathcal{L}_{c(t_\diamond)})\overline{v}\|_{L^2_{\overline{\mathbf{w}}}}\stackrel{\eqref{eqn:linearop}}{=}&\big\|(c(t_\diamond+s)\!-\!c(t_\diamond))\partial_x \overline{v}-2\partial_x((\phi_{c(t_\diamond+s)}-\phi_{c(t_\diamond)})\overline{v})\big\|_{L^2_{\overline{\mathbf{w}}}}\\
\leq & |c(t_\diamond+s)-c(t_\diamond)|\|\partial_x \overline{v}\|_{L^2_{\overline{\mathbf{w}}}}\\
&+2\|\partial_x(\phi_{c(t_\diamond+s)}-\phi_{c(t_\diamond)})\|_{L^\infty}\|\overline{v}\|_{L^2_{\overline{\mathbf{w}}}}\\
&+2\|\phi_{c(t_\diamond+s)}-\phi_{c(t_\diamond)}\|_{L^\infty}\|\partial_x\overline{v}\|_{L^2_{\overline{\mathbf{w}}}}.\end{align*}
Hence, we see that
\[\|(\mathcal{L}_{c(t_\diamond+s)}-\mathcal{L}_{c(t_\diamond)})\overline{v}\|_{L^2_{\overline{\mathbf{w}}}}\leq \tilde{C}_1 |c(t_\diamond+s)-c(t_\diamond)|\|\overline{v}\|_{H^1_{\overline{\mathbf{w}}}}\]
for some constant $\tilde{C}_1>0$, since $c\mapsto \phi_c+\partial_x \phi_c$ is Lipschitz from $\R$ to $L^\infty$.
Next, we apply \eqref{eqn:modulationestimate} at $t_\diamond+s$ with weight $\alpha(t_\diamond+s)\overline{\mathbf{w}}$ to obtain
\begin{align*}
   |\Omega_t(t_\diamond+s)| \leq & C_6\epsilon|f(\gamma(t_\diamond+s))|\big(1+\alpha(t_\diamond+s)^{-1/2}\overline{w}_+^{-1/2}\|v(t_\diamond+s)\|_{L^2_{\alpha(t_\diamond+s)\overline{\mathbf{w}}}}\big)\\
   &+C_6\|v(t_\diamond+s)\|^2_{L^2_{\alpha(t_\diamond+s)\overline{\mathbf{w}}}}.
\end{align*}
Substituting $v(t_\diamond,x)=\alpha^2(t_\diamond)\overline{v}(t_\diamond,\alpha(t_\diamond)x)$, we then find that
\begin{align*}
    &\|{\Omega_t(t_\diamond+s)}\partial_x (\overline{v}(t_\diamond+s)+\phi_{c(t_\diamond+s)})\|_{L^2_{\overline{\mathbf{w}}}}\\
    &\leq \tilde{C}_2\Big(\epsilon|f(\gamma(t_\diamond+s))|(1+\|\overline{v}(t_\diamond+s)\|_{L^2_{\overline{\mathbf{w}}}})+\|\overline{v}(t_\diamond+s)\|^2_{L^2_{\overline{\mathbf{w}}}}\Big)(1+\|\partial_x \overline{v}(t_\diamond+s)\|_{L^2_{\overline{\mathbf{w}}}})
\end{align*}
for some constant $\tilde{C}_2>0$.
Clearly, the term $\|{c_t(t_\diamond+s)}\partial_c\phi_{c(t_\diamond+s)}\|_{L^2_{\overline{\mathbf{w}}}}$ satisfies the same bound upon increasing $\tilde{C}_2>0$ if necessary. Lastly, we estimate
\begin{align}\label{eqn:nonlinear}
    \|N(\overline{v})\|_{L^2_{\overline{\mathbf{w}}}}=2\|\overline{v}\partial_x\overline{v}\|_{L^2_{\overline{\mathbf{w}}}}\leq 2\sqrt{2}\|\overline{v}\|_{H^1}\|\partial_x\overline{v}\|_{L^2_{\overline{\mathbf{w}}}},
\end{align}
using the continuous embedding $H^1(\R)\hookrightarrow L^\infty(\R)$.
\end{proof}

With this, we are equipped to prove \Cref{prop:short}. Another slight complication compared to \cite{pegoweinstein} is that $\overline{v}_t$ is not completely in the stable subspace characterized by \eqref{eqn:orthogonality}. This is a result of the fact that our construction of the modulation parameters ensures that $v_t$ satisfies the orthogonality conditions (see \eqref{eqn:vtorthogonal}), whereas $\overline{v}_t$ does not. We therefore decompose $\overline{v}_t$ using $I=P_c+Q_c$, where we recall that $P_c$ is the spectral projection corresponding to the 0-eigenvalue of $\mathcal{L}_c$ as defined in \eqref{eqn:spectral}. %See \Cref{app:old} for details.

\begin{proof}[Proof of \Cref{prop:short}]
The unscaled perturbation satisfies
\begin{align}\label{eqn:mild}
    \overline{v}(t_\diamond+r)=&e^{r\mathcal{L}_{c(t_\diamond)}}\overline{v}(t_\diamond)+\int_0^{r} e^{(r-s)\mathcal{L}_{c(t_\diamond)}}\big(P_{c(t_\diamond)}+Q_{c(t_\diamond)}\big)Y(t_\diamond,s,\overline{v};c,\Omega)\ \d s
\end{align}
for $r\in[0,\delta]$ and $t_\diamond\in [0,T-\delta]$, which is the mild form of \eqref{eqn:strong}. Since the orthogonality condition \eqref{eqn:orthogonality} ensures $P_{c(t_\diamond)}\overline{v}(t_\diamond)=0$, we have
\begin{align*}
    \|e^{r\mathcal{L}_{c(t_\diamond)}}\overline{v}(t_\diamond)\|_{H^1_{\overline{\mathbf{w}}}} \stackrel{\eqref{eqn:stabilityandsmoothing}}{\leq} M e^{-\beta r} \|\overline{v}(t_\diamond)\|_{H^1_{\overline{\mathbf{w}}}},
\end{align*}
for $\beta=\tfrac{4}{9}c_0\alpha_{\mathrm{max}}^{-2}w_{\mathrm{min}}$, where we use that
\[\tfrac{8}{9}c_0\alpha_{\mathrm{max}}^{-2}w_{\mathrm{min}}\leq\min\{\overline{w}_-(c(t_\diamond)-\overline{w}_-^2),\overline{w}_+(c(t_\diamond)-\overline{w}_+^2)\}.\]
To account for the projection onto the stable subspace, we use the stability and smoothing properties of $e^{t\mathcal{L}_{c(t_\diamond)}}Q_{c(t_\diamond)}$ to obtain
\begin{align*}
    &\left\|\int_0^{r} e^{(r-s)\mathcal{L}_{c(t_\diamond)}}Q_{c(t_\diamond)}Y(t_\diamond,s,\overline{v};c,\Omega)\ \d s\right\|_{H^1_{\overline{\mathbf{w}}}} \\
&\stackrel{\eqref{eqn:stabilityandsmoothing}}{\leq} M \int_0^\delta e^{-\beta(\delta-s)}(\delta-s)^{-1/2} \|Y(t_\diamond,s,\overline{v};c,\Omega)\|_{L^2_{\overline{\mathbf{w}}}} \d s.
\end{align*}
Writing
\[\epsilon_1=\sup_{s\in[0,\delta]}\Big(\|\overline{v}(t_\diamond+s)\|_{H^1_{\overline{\mathbf{w}}}}+
    \|\overline{v}(t_\diamond+s)\|_{H^1}+
    |c(t_\diamond+s)-c(t_\diamond)|\Big),\]
\Cref{cor:yts} can be used to derive that
\begin{align}\sup_{s\in[0,\delta]}\|Y(t_\diamond,s,\overline{v};c,\Omega)\|_{L^2_{\overline{\mathbf{w}}}}\leq &C_7\big(\epsilon_1+\epsilon \|f\|_\infty (1+\epsilon_1)+\epsilon_1^2\big)\sup_{s\in[0,\delta]}\|\overline{v}(t_\diamond+s)\|_{H^1_{\overline{\mathbf{w}}}}\nonumber\\
&+C_6\epsilon \sup_{s\in[0,\delta]}\big|f\big(\gamma (t_\diamond+s)\big)\big|.\label{eqn:ytsapplied}\end{align}
Furthermore, using $\int_0^\delta (\delta-s)^{-1/2} \d s =2\sqrt{\delta}$, we find that 
\begin{align*}
    \left\|\int_0^{r} e^{(r-s)\mathcal{L}_{c(t_\diamond)}}Q_{c(t_\diamond)}Y(t_\diamond,s,\overline{v};c,\Omega)\ \d s\right\|_{H^1_{\overline{\mathbf{w}}}}\leq &  \tilde{C}_1(\epsilon+ \epsilon_1)\sqrt{\delta} \sup_{s\in[0,\delta]}\|\overline{v}(t_\diamond+s)\|_{H^1_{\overline{\mathbf{w}}}}\\
     &+\tilde{C}_1\epsilon\sqrt{\delta}\sup_{s\in[0,\delta]}\big|f\big(\gamma (t_\diamond+s)\big)\big|
\end{align*}
for some constant $\tilde{C}_1>0$. To estimate the component in \eqref{eqn:mild} with the projection $P_{c(t_\diamond)}$, we recall that
\begin{align*}P_{c(t_\diamond)} Y(t_\diamond,s,\overline{v};c,\Omega)=&\big\langle Y(t_\diamond,s,\overline{v};c,\Omega), \eta_{c(t_\diamond)}^1 \big\rangle \partial_x \phi_{c(t_\diamond)} \\
&+\big\langle Y(t_\diamond,s,\overline{v};c,\Omega), \eta_{c(t)}^2\big\rangle \partial_{c} \phi_{c(t_\diamond)}\end{align*}
where $\mathcal{L}_{c(t_\diamond)} \partial_x\phi_{c(t_\diamond)}=0$ and $\mathcal{L}_{c(t_\diamond)} \partial_c \phi_{c(t_\diamond)}=\partial_x\phi_{c(t_\diamond)}$. Thus,
\begin{align*}
    e^{(r-s)\mathcal{L}_{c(t_\diamond)}}P_{c(t_\diamond)} Y(t_\diamond,s,\overline{v};c,\Omega)=&\big\langle Y(t_\diamond,s,\overline{v};c,\Omega), \eta_{c(t_\diamond)}^1 \big\rangle \partial_x \phi_{c(t_\diamond)} \\
    &+\big\langle Y(t_\diamond,s,\overline{v};c,\Omega), \eta_{c(t_\diamond)}^2\big\rangle \partial_c \phi_{c(t_\diamond)}\\
    &+ \big\langle Y(t_\diamond,s,\overline{v};c,\Omega), \eta_{c(t_\diamond)}^2\big\rangle(\delta-s)\partial_x \phi_{c(t_\diamond)},
\end{align*}
and we may estimate
\begin{align*}
    &\|e^{(r-s)\mathcal{L}_{c(t_\diamond)}}P_{c(t_\diamond)} Y(t_\diamond,s,\overline{v};c,\Omega)\|_{H^1_{\overline{\mathbf{w}}}}\\
    &\leq \| Y(t_\diamond,s,\overline{v};c,\Omega)\|_{L^2_{\overline{\mathbf{w}}}}\| \eta_{c(t_\diamond)}^1 \|_{L^2_{-{\overline{\mathbf{w}}}}} \|\partial_x \phi_{c(t_\diamond)} \|_{H^1_{\overline{\mathbf{w}}}}\\
    &\quad+\| Y(t_\diamond,s,\overline{v};c,\Omega)\|_{L^2_{\overline{\mathbf{w}}}}\|\eta_{c(t_\diamond)}^2\|_{L^2_{-{\overline{\mathbf{w}}}}} \|\partial_c \phi_{c(t_\diamond)}\|_{H^1_{\overline{\mathbf{w}}}}\\
    &\quad+ \| Y(t_\diamond,s,\overline{v};c,\Omega)\|_{L^2_{\overline{\mathbf{w}}}}\|\eta_{c(t_\diamond)}^2\|_{L^2_{-{\overline{\mathbf{w}}}}} (r-s)\|\partial_x \phi_{c(t_\diamond)}\|_{H^1_{\overline{\mathbf{w}}}}\\
    &\leq  \tilde{C}_2\| Y(t_\diamond,s,\overline{v};c,\Omega)\|_{L^2_{\overline{\mathbf{w}}}}(1+\delta-s),
\end{align*}
for some constant $\tilde{C}_2>0$. Integrating this inequality, it follows via \eqref{eqn:ytsapplied} that there exists a constant $\tilde{C}_3>0$ for which
\begin{align*}
     &\left\|\int_0^{r} e^{(r-s)\mathcal{L}_{c(t_\diamond)}}P_{c(t_\diamond)}Y(t_\diamond,s,\overline{v};c,\Omega)\ \d s\right\|_{H^1_{\overline{\mathbf{w}}}} \\ &\leq  \tilde{C}_3(\epsilon +\epsilon_1)(\delta+\delta^2)\sup_{s\in[0,\delta]}\|\overline{v}(t_\diamond+s)\|_{H^1_{\overline{\mathbf{w}}}}
    +\tilde{C}_3\epsilon(\delta+\delta^2) \sup_{s\in[0,\delta]}\big|f\big(\gamma (t_\diamond+s)\big)\big|.
\end{align*}
We collect that
\begin{align*}
    \|\overline{v}(t_\diamond+s)\|_{H^1_{\overline{\mathbf{w}}}}\leq & M e^{-\beta s}\|\overline{v}(t_\diamond)\|_{H^1_{\overline{\mathbf{w}}}}+\epsilon\big(\tilde{C}_1\sqrt{\delta}+\tilde{C}_3\delta+\tilde{C}_3\delta^2\big) \sup_{s\in[0,\delta]}\big|f\big(\gamma (t_\diamond+s)\big)\big|\\
    &+  (\epsilon+\epsilon_1)\big(\tilde{C}_1\sqrt{\delta}+\tilde{C}_3\delta+\tilde{C}_3\delta^2\big)\sup_{s\in[0,\delta]}\|\overline{v}(t_\diamond+s)\|_{H^1_{\overline{\mathbf{w}}}}
\end{align*}
for $s\in [0,\delta]$. The result follows by taking a supremum over $s\in [0,\delta]$, and choosing $\delta_4$ small enough such that $\sup_{s\in[0,\delta]}\|\overline{v}(t_\diamond+s)\|_{H^1_{\overline{\mathbf{w}}}}$ can be brought to the left.
\end{proof}
We conclude this section with the proof of \Cref{cor:short}. Essentially, we apply \Cref{lem:scaling} to translate \Cref{prop:short} to the rescaled frame. Recall that \hyperlink{hyp:timevarying}{C2} implies that $t\to \tfrac{w_+(t)}{{\alpha}(t)}$ is decreasing and $t\to \tfrac{w_-(t)}{{\alpha}(t)}$ is increasing on $[0,T]$. This guarantees that the weight-functions are also monotone in the original frame.

\begin{proof}[Proof of \Cref{cor:short}]
Setting $\overline{\mathbf{b}}=\frac{\mathbf{w}(t+\delta)}{{\alpha}(t+\delta)}$, we have
\begin{align*}\sup_{s\in [0,\delta]}\|\overline{v}(t+s)\|_{H^1_{\overline{\mathbf{b}}}}\leq& \sup_{s\in [0,\delta]}\|\overline{v}(t+s)\|_{H^1_{\mathbf{w}(t+s)/\alpha(t+s)}}
\leq \tilde{C}_1\sup_{s\in [0,\delta]}\|v(t+s)\|_{H^1_{\mathbf{w}(t+s)}}
,
\end{align*}
for some constant $\tilde{C}_1>0$, where we have used \Cref{lem:scaling} in the last step. If $\delta_5$ is small enough, then we can apply \Cref{prop:short} to obtain
\[\sup_{s\in[0,\delta]}\|\overline{v}(t+s)\|_{H^1_{\overline{\mathbf{b}}}}\leq C_4\Big( \|\overline{v}(t)\|_{H^1_{\overline{\mathbf{b}}}}+\epsilon \sup_{s\in[0,\delta]}\big|f\big(\gamma (t+s)\big)\big| (\sqrt{\delta}+\delta^2)\Big).\]
Choosing $C_5$ large enough, we have 
\begin{align}\sup_{s\in [0,\delta]}\|{v}(t\!+\!s)\|_{H^1_{{\alpha}(t+s) \overline{\mathbf{b}}}}\leq C_5\Big(\|{v}(t)\|_{H^1_{{\alpha}(t)\overline{\mathbf{b}}}}\!+\!\epsilon \sup_{s\in[0,\delta]}\big|f\big(\gamma (t\!+\!s)\big)\big|(\sqrt{\delta}\!+\!\delta^2)\Big)\label{eqn:equivalent}\end{align}
via \Cref{lem:scaling}. It follows that
\begin{align*}
    \|{v}(t+\delta)\|_{H^1_{\mathbf{w}(t+\delta)}}=&\|{v}(t+\delta)\|_{H^1_{{\alpha}(t+\delta) \overline{\mathbf{b}}}}\\
    \stackrel{\eqref{eqn:equivalent}}{\leq}& C_5\Big(\|{v}(t)\|_{H^1_{{\alpha}(t)\overline{\mathbf{b}}}}+\epsilon \sup_{s\in[0,\delta]}\big|f\big(\gamma (t+s)\big)\big|(\sqrt{\delta}+\delta^2)\Big)
\end{align*}
and the conclusion follows via \eqref{eqn:weightcondition} and \Cref{lem:observations}.
\end{proof}

\section{Long-time control}
\label{sec:longtime}
In this section, we establish control of the perturbation $v$ over long time 
intervals under the assumption that the forcing term is exponentially bounded. As a preparation, we fix the minimum weight $w_{\mathrm{min}}$ in \hyperlink{hyp:wmin}{S3}
and an intermediate weight $w_\infty\in(w_{\mathrm{min}},w)$ by writing\footnote{In principle, the constant $E_{\mathrm{max}}$ in \eqref{eqn:wmin} can be replaced by $E=\epsilon/\gamma$. For clarity, however, we work with $E_{\mathrm{max}}$ so that the weights do not depend on $\epsilon$ and $\gamma$.}
\begin{align}w_{\mathrm{min}}=we^{-4 C_6 E_{\mathrm{max}}(2+\delta_6) e^{1/2}}\quad \text{and} \quad w_\infty=we^{-2 C_6 E_{\mathrm{max}}(2+\delta_6) e^{1/2}},\label{eqn:wmin}\end{align}
using the constants $C_6$ and $\delta_6$ from \Cref{lem:estimates}. 
In particular, this choice only depends on \hyperlink{hyp:initial}{S1} and \hyperlink{hyp:alphabound}{S2}. In addition, we introduce
functions $W_-:\R^+\to [w_{\mathrm{min}},w_\infty)$ and $W_+:\R^+\to (w_\infty, w]$ through 
\begin{align}W_-(t)=w_{\mathrm{min}}(\tfrac{w_\infty}{w_{\mathrm{min}}})^{1-e^{-  t}}, \quad W_+(t)=w(\tfrac{w_\infty}{w})^{1-e^{- t}}\label{eqn:weightdef};\end{align}
see \Cref{fig:weightgraph}. 
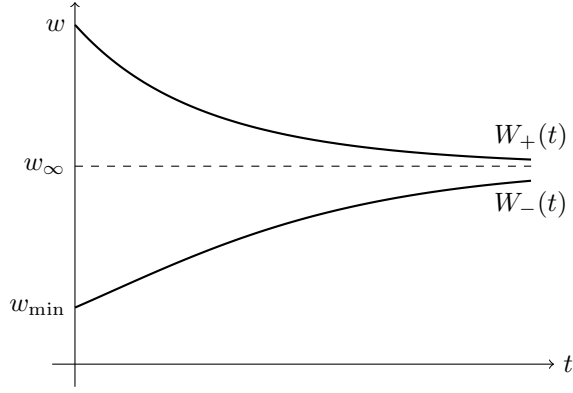
\begin{figure}[t]
    \centering
\begin{tikzpicture}[scale=1.5]
    % Constants
    \def\aZero{0.5}  % Adjust as needed
    \def\aInfinity{1.75}  % Adjust as needed
    \def\aMax{3}  % Adjust as needed
    \def\gammaValue{0.7}  % Adjust as needed
    \def\timeMax{4}  % Adjust as needed

    % Function a(t)
    \draw[thick, domain=0:\timeMax, samples=100] plot (\x, {\aZero*pow(\aInfinity/\aZero,1-exp(-\gammaValue*\x))}) node[below] {$W_-(t)$};

    % Function b(t)
    \draw[thick, domain=0:\timeMax, samples=100] plot (\x, {\aMax*pow(\aInfinity/\aMax,1-exp(-\gammaValue*\x))}) node[above] {$W_+(t)$};
    
    % Axes
    \draw[->] (-0.2,0) -- (\timeMax+0.2,0) node[right] {$t$};
    \draw[->] (0,-0.2) -- (0,\aMax+0.2) ;

    % % Time labels
    % \foreach \i in {1,...,\timeMax}
    %     \draw (\i,-0.1) -- (\i,0.1) node[above] {\i};

    % Value labels
        \draw (0,\aZero) node[left] {$w_{\mathrm{min}}$};
        \draw (0,\aMax) node[left] {$w$};
        \draw [dashed] (0,\aInfinity) -- (4,\aInfinity);
        \draw (0,\aInfinity) node[left] {$w_\infty$};
        % \draw[<->] (-0.1,\aMax-0.1) -- (-0.1,\aInfinity+0.1);
        % \draw[<->] (-0.1,\aInfinity-0.1) -- (-0.1,\aZero+0.1);
        % \draw (-0.1,0.5*\aMax+0.5*\aInfinity) node[left] {$O(\epsilon/\gamma+\epsilon_1^2/\gamma)$};
        % \draw (-0.1,0.5*\aInfinity+0.5*\aZero) node[left] {$O(\epsilon/\gamma+\epsilon_1^2/\gamma)$};
\end{tikzpicture}
    \caption{Graph of the weight-functions $W_-$ and $W_+$
    defined in \eqref{eqn:weightdef}. We remark that $w_{\mathrm{min}}$ and $w_\infty$ defined in \eqref{eqn:wmin} approach $w$ upon letting $E_{\mathrm{max}}\downarrow 0$.}
    \label{fig:weightgraph}
\end{figure}
We then claim that the weight-function $\mathbf{w}(t) = \big(w_-(t), w_+(t) \big)$ defined by
\begin{equation}
    \label{eq:defn:w:pm}
    w_-(t)=W_-(\gamma t), \quad w_+(t)=W_+(\gamma t)
\end{equation}
%$\mathbf{w}(t)=\big(W_-(\gamma t),W_+(\gamma t)\big)$
satisfies \hyperlink{hyp:timevarying}{C2}, provided that the following conditions are met.

%At this point, we construct time-dependent weights to match \hyperlink{hyp:timevarying}{C2}, where we fix the constant
%\[w_{\mathrm{min}}=ae^{-4 C_6 (E_{\mathrm{max}}(1+\delta_6)+1 )e^{1/2}},\]
%depending only on \hyperlink{hyp:initial}{S1} and \hyperlink{hyp:alphabound}{S2}. Here, $C_6$ and $\delta_6$ are the constants from \Cref{lem:estimates}. Throughout this section, we assume the following setting.
\begin{itemize}
\item[\textbf{C3}]{
  \phantomsection\hypertarget{hyp:epsilon}{}\textit{
  We have $p\in[0,\frac{1}{4})$ together with $\epsilon \in (0,1]$ and $\gamma > 0$. In addition, we have $\delta=\epsilon^{-p}$ together with the
inequalities
%Given $\epsilon, \gamma>0$ and $p\in[0,\frac{1}{4})$, we set $\delta=\epsilon^{-p}$. We assume that 
\begin{itemize}
\item[$\bullet$] $\delta\leq \gamma^{-1}$; 
\item[$\bullet$] $\gamma\leq \tfrac{1}{3}\alpha^{-3}_{\mathrm{max}}w_{\mathrm{min}}(c_0-w_{\mathrm{min}}^2)$;
\item[$\bullet$]  $\epsilon/\gamma\leq E_{\mathrm{max}}$.
\end{itemize}}}
\end{itemize}
% \begin{condition}\label{hyp:epsilon}
% We have $p\in[0,\frac{1}{4})$ together with $\epsilon > 0$ and $\gamma > 0$. In addition, we have $\delta=\epsilon^{-p}$ together with the
% inequalities
% %Given $\epsilon, \gamma>0$ and $p\in[0,\frac{1}{4})$, we set $\delta=\epsilon^{-p}$. We assume that 
% \begin{itemize}
% \item $\delta\leq \gamma^{-1}$; 
% \item $\gamma\leq \tfrac{1}{3}\alpha^{-3}_{\mathrm{max}}w_{\mathrm{min}}(c_0-w_{\mathrm{min}}^2)$;
% \item  $\epsilon/\gamma\leq E_{\mathrm{max}}$.
% \end{itemize}
% \end{condition}
\begin{itemize}
\item[\textbf{C4}]{
  \phantomsection\hypertarget{hyp:exponential}{}\textit{The forcing term $f$ is continuous and bounded as 
 \[|f(t)| \leq   e^{- t}, \quad t\geq 0.\]
 }}
\end{itemize}
% \begin{condition}\label{hyp:exponential}
%    The forcing term $f$ is measureable and bounded as 
% \[|f(t)| \leq   e^{- t}, \quad t\in[0,T].\]
% \end{condition}
\begin{itemize}
\item[\textbf{C5}]{
  \phantomsection\hypertarget{hyp:Tproperties2}{}\textit{Given $T>0$, the function ${v}\in  C([0, T],H^2) \cap C^1([0, T],H^{-1})$ satisfies \eqref{eqn:vmod}, while $\alpha \in C^1([0,T],\R)$ solves \eqref{eqn:modulation}. In addition, we have the inclusion
    \[\alpha(t)\in [\alpha_{\mathrm{min}},\alpha_{\mathrm{max}}], \quad t\in[0,T].\]
 }}
\end{itemize}
% \begin{condition}\label{hyp:Tproperties2}
% Given $T>0$, let ${v}\in  C([0, T],H^2) \cap C^1([0, T],H^{-1})$ solve \eqref{eqn:vmod} and $\alpha \in C^1([0,T],\R)$ solve \eqref{eqn:modulation}. We assume that
%     \[\alpha(t)\in [\alpha_{\mathrm{min}},\alpha_{\mathrm{max}}], \quad t\in[0,T].\]
% \end{condition}

%With these assumptions in place, we construct weight-functions $\mathbf{w}(t)=(w_-(t),w_+(t))$ as follows.

\begin{lemma}[See \Cref{app:weight}]\label{lem:weightcondition}
Assume \hyperlink{hyp:initial}{S1}, \hyperlink{hyp:alphabound}{S2}, and the choice \eqref{eqn:wmin} for \hyperlink{hyp:wmin}{S3}.
If $\hyperlink{hyp:epsilon}{C3}$ and \hyperlink{hyp:exponential}{C4} are satisfied, and \hyperlink{hyp:Tproperties2}{C5} holds for some $T>0$,
%, and $\epsilon, \gamma>0$ and $p\in[0,\frac{1}{4})$ satisfy $\hyperlink{hyp:epsilon}{C3}$, 
then the bound
\begin{align}e^{ \gamma t}\|{v}
(t)\|_{L^2_{w_\infty}}\leq \min\{\delta_6 w_\infty^{1/2},(E_{\mathrm{max}}\gamma)^{1/2}\}, \quad t \in [0,T],\label{eqn:assumption}\end{align}
implies that the functions $w_-$ and $w_+$ defined in \eqref{eq:defn:w:pm}
%\[w_-(t)=W_-(\gamma t), \quad w_+(t)=W_+(\gamma t)\]
%with $W_-, W_+$ given by \eqref{eqn:weightdef} 
satisfy \hyperlink{hyp:timevarying}{C2}.
\end{lemma}
%We summarize the construction of the weight-functions $w_-$ and $w_+$ in \Cref{fig:weightgraph}.

Our main result here provides exponential control for
the perturbation $v$ measured with respect to the time-varying 
spatial weight-functions $\mathbf{w}(t)$ defined in \eqref{eq:defn:w:pm}. It requires a priori
control\footnote{This control is established in the proof of \Cref{thm:bigthm}.} over the terms in the expression
%Let us introduce for each $\epsilon,\gamma,t\geq 0$ the constant
\begin{align}K_{\epsilon,\gamma}(t)\!=\!\epsilon\!+\!\frac{\epsilon^{1+p}}{\gamma }\!+\!\sup_{s\in[0,t]} \Big(e^{ \gamma s}\|{v}
(s)\|_{H^1_{\mathbf{w}(s)}}\!+\!\|{v}
(s)\|_{H^1}\!+\!\gamma^{-1}e^{2 \gamma s}\|{v}
(s)\|^2_{H^1_{\mathbf{w}(s)}}\Big),\label{eqn:kdef}\end{align}
which we have introduced for notational convenience. We note in particular that condition \eqref{eqn:assumption} is satisfied if $K_{\epsilon,\gamma}(t)$ is sufficiently small.
%Our result for control of the perturbation measured in time-varying weights then is the following.

\begin{proposition}[Long-time control]\label{prop:duhamel}
%Assuming \hyperlink{hyp:initial}{S1} - \hyperlink{hyp:wmin}{S3},
Assuming \hyperlink{hyp:initial}{S1}, \hyperlink{hyp:alphabound}{S2}, and the choice \eqref{eqn:wmin} for \hyperlink{hyp:wmin}{S3},
there exist constants $\delta_{8}, C_8>0$ so that the following holds true. If $\hyperlink{hyp:epsilon}{C3}$ and \hyperlink{hyp:exponential}{C4} are satisfied, and \hyperlink{hyp:Tproperties2}{C5} holds for some $T>0$,
%, and $\epsilon\in (0,1], \gamma>0$ and $p\in[0,\frac{1}{4})$ satisfy $\hyperlink{hyp:epsilon}{C3}$, 
then the bound
\begin{align}K_{\epsilon,\gamma}(T)\leq \delta_{8}\label{eqn:assdelta6}\end{align}
% \begin{align}
% e^{b\epsilon t}\|{v}
% (t)\|_{H^1_{\mathbf{w}(t)}}\leq &\epsilon_1\label{eqn:bound}\\
% \|{v}(t)\|_{H^1}\leq &\epsilon_2 \nonumber\end{align} 
implies
\[\sup_{0\leq t \leq T}e^{  \gamma t}\|{v}(t)\|_{H^1_{\mathbf{w}(t)}}\leq C_8\big(\|\overline{v}_*\|_{H^1}+\|\overline{v}_*\|_{H^1_{w}}+\epsilon ^{1-2p}\big).\]
\end{proposition}
The proof of \Cref{prop:duhamel} is based on the evolution equation \eqref{eqn:vmod}, in which we isolate the precarious term $\frac{{\alpha}_t}{{\alpha}}x\partial_x v$ as 
% \begin{align*}
%     \tilde{v}_t=\mathcal{L}_{c_0}\tilde{v}+N(\tilde{v})+R(\alpha,\tilde{v}).
% \end{align*}
% where 
% \begin{align*}
%     R(\alpha,\tilde{v})=\frac{\alpha_t}{\alpha}(2+x\partial_x)[\phi_{c_0}+\tilde{v}]+\frac{\tilde{\Omega}_t}{\alpha}\partial_x[\phi_{c_0}+\tilde{v}]+\alpha^3\tilde{f}(t)[\phi_{c_0}+\tilde{v}].
% \end{align*}
\begin{align}\label{eqn:strong2}
    {v}_t=\alpha^{-3}\mathcal{L}_{c_0}{v}+\frac{{\alpha}_t}{{\alpha}}x\partial_x{v}+M(t,v;\alpha,\Omega)
\end{align}
with
\begin{align*}
    M(t,v;\alpha,\Omega)\!=\!\alpha^{-3}N({v})\!+\!\frac{{\alpha}_t}{{\alpha}}\big((2\!+\!x\partial_x)\phi_{c_0}\!+\!2{v}\big)\!+\!\frac{{\Omega}_t}{{\alpha}}\partial_x(\phi_{c_0}\!+\!{v})
   \! +\!\epsilon{f}(\gamma t)(\phi_{c_0}+{v}).
\end{align*}
In preparation, we show that $M(t,v;\alpha,\Omega)$ can be controlled in terms of $v$. Additionally, we bound the fluctuations of $c(t)$ over a time-step $\delta$ in terms of $v$ and $\delta$. Recalling that $\delta_6$ is the constant introduced in \Cref{lem:estimates}, we prove the following.

\begin{lemma}\label{lem:Mestimate}
Assuming \hyperlink{hyp:initial}{S1}, \hyperlink{hyp:alphabound}{S2}, 
and the choice \eqref{eqn:wmin} for \hyperlink{hyp:wmin}{S3}, 
there exist a constant $C_9>0$ so that the following holds true for each $\epsilon,\gamma>0$. If \hyperlink{hyp:Tproperties2}{C5} holds for some $T>0$, then for each $\mathbf{b}=(b_-,b_+)\in \R^2$ with $b_-, b_+\in[w_{\mathrm{min}},\sqrt{c_*}/3]$ and $t \in [0,T]$, the bound
$\|v(t)\|_{L^2_{\mathbf{b}}}\leq \delta_6 b_+^{1/2} $ implies
    \begin{align}
        \|M(t,v(t);\alpha,\Omega)\|_{L^2_{\mathbf{b}}}\leq& C_9 \epsilon|{f}(\gamma t)|(1+\|{v}(t)\|_{H^1_{\mathbf{b}}})\|{v}(t)\|_{H^1_{\mathbf{b}}}\nonumber\\
        &+C_9\Big(\|{v}(t)\|_{H^1_{\mathbf{b}}}+\|{v}(t)\|_{H^1_{\mathbf{b}}}^2 +\|{v}(t)\|_{H^1}\Big)\|{v}(t)\|_{H^1_{\mathbf{b}}} \nonumber\\
        &+C_9\epsilon|{f}(\gamma t)|.\label{eqn:mvtestimate}
    \end{align}
    If, furthermore, $\delta>0$ satisfies $t+\delta\leq T$ and
    \[\sup_{s\in [0,\delta]}\|v(t+s)\|_{L^2_{\mathbf{b}}}\leq \delta_6 b_+^{1/2},\] 
    then
    \begin{align}
       \sup_{0\leq s
       \leq \delta} |c(t\!+\!s)\!-\!c(t)|\leq C_9\delta \Big(\epsilon (1\!+\!\sup_{s\in [0,\delta]}\|v(t\!+\!s)\|_{L^2_{\mathbf{b}}})\!+\!\sup_{s\in [0,\delta]}\|v(t\!+\!s)\|_{L^2_{\mathbf{b}}}^2 \Big).\label{eqn:ccontrol}
    \end{align}
\end{lemma}
\begin{proof}
Writing $v(t)=v$, estimate \eqref{eqn:mvtestimate} follows by computing
    \begin{align*}
        \|M(t,v;\alpha,\Omega)\|_{L^2_{\mathbf{b}}}\leq& \alpha_{\mathrm{min}}^{-3}\|N({v})\|_{L^2_{\mathbf{b}}}+2\left|\frac{{\alpha}_t}{{\alpha}}\right|\big\|(1+\tfrac{1}{2}x\partial_x)\phi_{c_0}+{v}\big\|_{L^2_{\mathbf{b}}}\\
        &+\left|\frac{{\Omega}_t}{{\alpha}}\right|\|\partial_x(\phi_{c_0}+{v})\|_{L^2_{\mathbf{b}}}+\epsilon|{f}(\gamma t)|\|\phi_{c_0}+{v}\|_{L^2_{\mathbf{b}}},
    \end{align*}
    estimating $\|N({v})\|_{L^2_{\mathbf{b}}}$ as in \eqref{eqn:nonlinear}, and applying the estimate \eqref{eqn:modulationestimate}.
To prove \eqref{eqn:ccontrol}, we estimate 
\[\sup_{s\in [0,\delta]}|{c}(t+s)-{c}(t)|
\leq \int_t^{t+\delta}|{c}_{t}(t^\prime)|\d t^\prime,\] and write
     \[\epsilon_1=\sup_{s\in [0,\delta]}\|v(t+s)\|_{L^2_{\mathbf{b}}}.\]
Using $
    c_t=-2c_0\alpha^{-3} \alpha_t$
and \eqref{eqn:modulationestimate}, we obtain
% \[-\alpha_t/\alpha\leq C \epsilon  \epsilon_1 e^{-(b+\lambda )\epsilon t}+C\epsilon_1e^{-b\epsilon t}\lesssim \epsilon_1 e^{-b\epsilon t}.\]
% Similarly
\begin{align*}\sup_{s\in [0,\delta]}|{c}(t+s)-{c}(t)|
\leq & 2c_0 \alpha^{-3}_{\mathrm{min}}C_6 \int_t^{t+\delta} \Big(\epsilon|{f}(\gamma t^\prime)|(1+b_+^{-1/2}\epsilon_1)+\epsilon_1^2 \Big)\d t^\prime \\
\leq & 2c_0 \alpha^{-3}_{\mathrm{min}}C_6\delta\big(\|f\|_{\infty}\epsilon (1+w_{\mathrm{min}}^{-1/2}\epsilon_1)+\epsilon_1^2 \big). \qedhere  \end{align*}
\end{proof}
As a final preparation for the proof of \Cref{prop:duhamel}, we establish control of $v(t)$ in the space $H^1_{\mathbf{w}(t+\delta/2)}$. We do so based on the integral equation
\begin{align}
    {v}(t)=&e^{\mathcal{L}_{c_0}\int_0^t \alpha^{-3}(t^\prime)\d t^\prime}\overline{v}_0 \nonumber\\
    &+\int_0^t e^{\mathcal{L}_{c_0}\int_s^t \alpha^{-3}(t^\prime)\d t^\prime}\Big(\frac{{\alpha}_t(s)}{{\alpha}(s)}x\partial_x{v}(s)+M(s,v(s);\alpha,\Omega)\Big) \d s,\label{eqn:mild2}
\end{align}
which holds for all $t\geq 0$ and is the mild form of \eqref{eqn:strong2}. 
Note that the weight is evaluated at time $t+\delta/2$, whereas $v$ is evaluated at $t$. This gap allows us to estimate the $x\partial_x v$ term in \eqref{eqn:mild2} without introducing a singularity. 
\begin{lemma}\label{lem:duhamel}
    Assuming \hyperlink{hyp:initial}{S1}, \hyperlink{hyp:alphabound}{S2}, and the choice \eqref{eqn:wmin} for \hyperlink{hyp:wmin}{S3},
    %Assuming \hyperlink{hyp:initial}{S1} - \hyperlink{hyp:wmin}{S3}, 
    there exists a constant $C_{10}>0$ so that the following holds true. If 
    $\hyperlink{hyp:epsilon}{C3}$ and \hyperlink{hyp:exponential}{C4} are satisfied, and 
    \hyperlink{hyp:Tproperties2}{C5} holds for some $T>0$,
    %\hyperlink{hyp:exponential}{C4} and \hyperlink{hyp:Tproperties2}{C5} hold for some $T>0$, and $\epsilon\in (0,1], \gamma>0$ and $p\in[0,\frac{1}{4})$ satisfy $\hyperlink{hyp:epsilon}{C3}$, 
    then for each $t\in[0,T]$, the bound
    \[\sup_{s\in[0,t]}e^{\gamma s} \|{v}
(s)\|_{H^1_{\mathbf{w}(s)}}\leq \min\{\delta_6 w_\infty^{1/2},1\} \]
 implies
\begin{align*}
    e^{ \gamma t}\| {v}(t)\|_{H^1_{\mathbf{w}(t+\delta/2)}}\leq & C_{10}K_{\epsilon,\gamma}(t) \sup_{s\in[0,t]}e^{ \gamma s}\|{v}(s)\|_{H^1_{\mathbf{w}(s)}}+C_{10}\big(\|\overline{v}_*\|_{H^1}+\|\overline{v}_*\|_{H^1_w}+\epsilon \big),
\end{align*}
where $K_{\epsilon,\gamma}(t)$ is defined in \eqref{eqn:kdef}.
\end{lemma}
\begin{proof}% [Proof of \Cref{lem:duhamel}]
Pick $t\in[0,T]$ and set
\[\epsilon_1=\sup_{s\in[0,t]} e^{ \gamma s}\|{v}
(s)\|_{H^1_{\mathbf{w}(s)}} \quad \text{and} \quad \epsilon_2=\sup_{s\in[0,t]}\|{v}
(s)\|_{H^1}.\]
Using \eqref{eqn:mild2}, we estimate
\begin{align}\label{eqn:forv0}
    \| v(t)\|_{H^1_{\mathbf{w}(t+\delta/2)}}\stackrel{\eqref{eqn:stabilityandsmoothing}}{\leq}Me^{-\beta t}\|\overline{v}_0\|_{H^1_{\mathbf{w}(t+\delta/2)}}+M\alpha_{\mathrm{min}}^{3/2}(I_1+I_2),
\end{align}
for $\beta =\tfrac{1}{2}\alpha^{-3}_{\mathrm{max}}w_{\mathrm{min}}(c_0-w_{\mathrm{min}}^2)$ with
\begin{align}
    I_1=&\int_0^t e^{-\beta(t-s)}(t-s)^{-1/2}\left|\frac{{\alpha}_t(s)}{{\alpha}(s)}\right| \|x\partial_x{v}(s)\|_{L^2_{\mathbf{w}(t+\delta/2)}} \d s;\label{eqn:I1}\\
    I_2=&\int_0^t e^{-\beta(t-s)}(t-s)^{-1/2}\|M(s,v(s);\alpha,\Omega)\|_{L^2_{\mathbf{w}(t+\delta/2)}}\d s.\label{eqn:I2}
\end{align}
Here, we have used that $\int_s^t\!\! \alpha^{-3}(t^\prime)\d t^\prime\!\geq\! \alpha^{-3}_{\mathrm{max}}(t\!-\!s)$. Using \Cref{lem:observations} and \Cref{lem:implicit}, we estimate 
\begin{align*}
    \|\overline{v}_0\|_{H^1_{\mathbf{w}(t+\delta/2)}}\leq \|\overline{v}_0\|_{H^1_{\mathbf{w}(0)}}\leq \|\overline{v}_0\|_{H^1_{w}}+\|\overline{v}_0\|_{H^1}\leq C_2(\|\overline{v}_*\|_{H^1}+\|\overline{v}_*\|_{H^1_w})
\end{align*}
in \eqref{eqn:forv0}. Focusing on $I_1$, we use \Cref{lem:observations} to estimate
\begin{align}\label{eqn:xforI1}
    \|x\partial_x{v}(s)\|_{L^2_{\mathbf{w}(t+\delta/2)}}\leq& e^{-1}  q(t,s,\delta)\|{v}(s)\|_{H^1_{\mathbf{w}(s)}} ,
\end{align}
where
\begin{align}q(t,s,\delta):=\tfrac{1}{w_-(t+\delta/2)-w_-(s)}+\tfrac{1}{w_+(s)-w_+(t+\delta/2)}.\label{eqn:q}\end{align}
Furthermore, we use \Cref{lem:estimates} to estimate
\begin{align}\left|\frac{{\alpha}_t(s)}{{\alpha}(s)}\right|\stackrel{\eqref{eqn:modulationestimate}}{\leq} &\alpha_{\mathrm{min}}^{-1}C_6\Big( \epsilon|f(\gamma s)| (1+w_{\infty}^{-1/2}\|v(s)\|_{L^2_{\mathbf{w}(s)}})+\|v(s)\|^2_{L^2_{\mathbf{w}(s)}}\Big)\nonumber\\
\leq &\alpha_{\mathrm{min}}^{-1}C_6\big( \epsilon e^{-\gamma s} (1+w_{\infty}^{-1/2}\epsilon_1)+\epsilon_1^2e^{-2\gamma s}\big)\nonumber \\
\leq & \tilde{C}_1 (\epsilon+\epsilon_1^2) e^{-\gamma s}\label{eqn:modforI1}\end{align}
for some constant $\tilde{C}_1>0$. Substituting \eqref{eqn:xforI1} and \eqref{eqn:modforI1} into \eqref{eqn:I1} yields
\begin{align*}
    e^{\gamma t}I_1\!\leq\! \tilde{C}_1 e^{-1}(\epsilon\!+\!\epsilon_1^2) \int_0^t  e^{-\gamma s} q(t,s,\delta)e^{-(\beta-\gamma)(t\!-\!s)}(t-s)^{-1/2}\d s \sup_{s\in[0,t]}e^{\gamma s}\|{v}(s)\|_{H^1_{\mathbf{w}(s)}},
\end{align*}
where we have used 
\begin{align}e^{\gamma t}e^{-\beta(t-s)}=e^{-(\beta-\gamma)(t-s)}e^{\gamma s}.\label{eqn:exptrick}\end{align} Invoking \Cref{lem:weight} to estimate 
\[\sup_{s\in[0,t]}e^{-\gamma  s}|q(t,s,\delta)|\leq \frac{\tilde{C}_2}{\gamma\delta}\] 
for some $\tilde{C}_2 > 0$, we find
\begin{align*}
    e^{\gamma t}I_1\leq & \tilde{C}_3 (\gamma \delta)^{-1}(\epsilon+\epsilon_1^2)\sup_{s\in[0,t]}e^{\gamma s}\|{v}(s)\|_{H^1_{\mathbf{w}(s)}}
\end{align*}
where
\[\tilde{C}_3=\tilde{C}_1\tilde{C}_2e^{-1}\int_0^\infty  e^{-(\beta-\gamma)s}s^{-1/2}\d s<\infty. \]
Recalling that $\delta=\epsilon^{-p}$ per \hyperlink{hyp:epsilon}{C3}, we have thus shown that
\[e^{\gamma t}I_1\leq \tilde{C}_3K_{\epsilon,\gamma}(t) \sup_{s\in[0,t]}e^{ \gamma s}\|{v}(s)\|_{H^1_{\mathbf{w}(s)}}.\]

Turning to $I_2$, we find via \eqref{eqn:exptrick} and \Cref{lem:observations} that 
\begin{align*}
    e^{\gamma t}I_2\leq &\int_0^t e^{-(\beta-\gamma)(t-s)}(t-s)^{-1/2}  e^{\gamma s}\|M(s,v(s);\alpha,\Omega)\|_{L_{\mathbf{w}(s)}^2} \d s.
\end{align*}
Applying \Cref{lem:Mestimate} yields
\begin{align*}
        \|M(s,v(s);\alpha,\Omega)\|_{L^2_{\mathbf{w}(s)}}\leq C_9 \big(\epsilon (1+\epsilon_1)+\epsilon_1+\epsilon_2+\epsilon_1^2 \big)\|{v}(s)\|_{H^1_{\mathbf{w}(s)}}
        +C_9\epsilon e^{-\gamma s}
    \end{align*}
so that 
\begin{align*}
    e^{\gamma t}I_2\leq
    \tilde{C}_4 (\epsilon +\epsilon_1+\epsilon_2) \sup_{s\in[0,t]}e^{\gamma s}\|{v}(s)\|_{H^1_{\mathbf{w}(s)}}+\tilde{C}_4 \epsilon
\end{align*}
for some $\tilde{C}_4>0$. The conclusion follows by setting 
$C_{10}=M\alpha_{\mathrm{min}}^{3/2}(\tilde{C}_3+\tilde{C}_4).$
% \begin{align*}
%     e^{b\epsilon t}\| {v}(t)\|_{H^1_{\mathbf{w}(t+\delta/2)}}\lesssim (\epsilon_1+\epsilon_2+\delta^{-1})\sup_{s\in[0,t]}e^{b\epsilon s}\|{v}(s)\|_{H^1_{\mathbf{w}(s)}}+\epsilon.
% \end{align*}
\end{proof}

We finally move on to the proof of \Cref{prop:duhamel}. We first control the fluctuations of $c(t)$ during the time-step $\delta$ through \Cref{lem:Mestimate}. This allows us to use the short-time result \Cref{cor:short}, after which we apply the long-time result \Cref{lem:duhamel}.

\begin{proof}[Proof of \Cref{prop:duhamel}]
We collect from \eqref{eqn:assdelta6} that
\begin{align*}
   \Big( \epsilon+\sup_{s\in [0,\delta]}\big(\|{v}
(t+s)\|_{H^1_{\mathbf{w}(t+s)}}+\|{v}
(t+s)\|_{H^1}\big)\Big)\delta^2\leq &\delta_{8}^{1-2p}.
\end{align*}
Furthermore, \Cref{lem:Mestimate} applied with weight $w_\infty$ provides the bound
\begin{align*}
\sup_{s\in [0,\delta]}|c(t\!+\!s)\!-\!c(t)|\delta^2\leq& C_9\delta^3 \Big(\epsilon (1+\sup_{s\in [0,\delta]}\|v(t+s)\|_{L^2_{w_\infty}})+\sup_{s\in [0,\delta]}\|v(t+s)\|_{L^2_{w_\infty}}^2 \Big)\\
\leq & C_9 \big(\delta_{8}^{1-3p} (1+\delta_{8})+\delta_{8}^{2-3p} \big).
\end{align*}

Picking $\delta_{8}$ small enough, we may apply \Cref{cor:short} at time $t$, which yields
\[e^{\gamma(t+\delta)}\sup_{s\in [0,\delta]}\|{v}(t+s)\|_{H^1_{\mathbf{w}(t+s)}}\leq C_5 \big(e^{\gamma(t+\delta)}\|{v}(t)\|_{H^1_{\mathbf{w}(t+\delta/2)}}+2\epsilon e^{\gamma(t+\delta)}e^{-\gamma t}\delta^2\big).\]
Applying \Cref{lem:duhamel}, we find
\begin{align*}e^{\gamma(t+\delta)}\sup_{s\in [0,\delta]}\|{v}(t+s)\|_{H^1_{\mathbf{w}(t+s)}}\leq &  \tilde{C}_1 K_{\epsilon,\gamma}(t) \sup_{s\in[0,t]}e^{\gamma s}\|{v}(s)\|_{H^1_{\mathbf{w}(s)}}\\
&+\tilde{C}_1\big(\|\overline{v}_*\|_{H^1}+\|\overline{v}_*\|_{H^1_w}+\epsilon+ \epsilon\delta^2\big),\end{align*}
where $\tilde{C}_1=C_5(C_{10}+2)e$. Taking the supremum over $t\in [0,T-\delta]$ and extending the supremum on the right-hand side
shows that
\begin{align*}
    \sup_{0\leq t \leq T}e^{\gamma t}\|{v}(t)\|_{H^1_{\mathbf{w}(t)}}
    \leq&\tilde{C}_1K_{\epsilon,\gamma}(T)\sup_{0\leq t \leq T}e^{\gamma t}\|{v}(t)\|_{H^1_{\mathbf{w}(t)}}\\
&+\tilde{C}_1\big(\|\overline{v}_*\|_{H^1}+\|\overline{v}_*\|_{H^1_w}+\epsilon+ \epsilon\delta^2\big).\end{align*}
% We can extend the supremum on the lefthandside via \Cref{cor:short} at time $0$:
% \[\sup_{0\leq t \leq \delta}\|{v}(t)\|_{H^1_{\mathbf{w}(t)}}\lesssim \epsilon (\sqrt{\overline{\delta}}\epsilon ^{-p/2 }+\overline{\delta}^2\epsilon ^{-2p}).\]
For $\delta_{8}$ small enough, we have $\tilde{C}_1 K_{\epsilon,\gamma}(T)<1/2$, and hence
\begin{equation*}
    \sup_{0\leq t \leq T}e^{\gamma t}\|{v}(t)\|_{H^1_{\mathbf{w}(t)}}
    \leq 2\tilde{C}_1\big(\|\overline{v}_*\|_{H^1}+\|\overline{v}_*\|_{H^1_w}+\epsilon+\epsilon^{1-2p}\big).   \qedhere
\end{equation*}
%and an application of .
\end{proof}

\section{Energy evolution}
\label{sec:energy}
In this section, we establish control of the perturbation in the unweighted spaces $L^2$ and $H^1$. Since the latter norm appears explicitly
in \eqref{eqn:kdef}, this is a crucial step towards applying
the long-time results developed in \Cref{sec:longtime}.
%since

%complements the main result of \Cref{sec:longtime}, \Cref{prop:duhamel}, which asserts control over $\|v(t)\|_{H^1_{\mathbf{w}(t)}}$ provided that (among other conditions) the perturbation $v(t)$ remains small in the unweighted $H^1$-norm.

Using the decomposition \eqref{eqn:unscaled} and the fact that $\langle \overline{v}(t),\phi_{c(t)}\rangle=0$, we observe via Pythagoras' theorem that
\begin{align}
    \|\overline{v}(t)\|^2_{L^2}=\|u(t,\cdot+\xi(t))\|_{L^2}^2-\|\phi_{c(t)}\|_{L^2}^2=\|u(t)\|_{L^2}^2-\|\phi_{c(t)}\|_{L^2}^2. \label{eqn:pyth}
    %=e^{2\int_0^t f(s)  \d s}6c_0^{3/2}-6c^{3/2}(t)=6(c_*^{3/2}(t)-c^{3/2}(t)),
\end{align}
We set out to control the $L^2$-norm of the perturbation $\overline{v}$
by estimating the time-derivative of \eqref{eqn:pyth}. The key point is that, in the right-hand side of \eqref{eqn:below} below, the term $\|\overline{v}(t)\|_{L^2}^2$ comes with an integrable factor $\epsilon |f(\gamma t)|$, and the remaining terms are all integrable and small. Invoking Gr\"{o}nwall's inequality, this can subsequently
be used to establish the desired control. % over the $L^2$-norm of $\overline{v}$.
We recall that $\delta_6$ is the constant introduced in \Cref{lem:estimates}.
\begin{lemma}\label{lem:L2derivative}
Assuming \hyperlink{hyp:initial}{S1}, \hyperlink{hyp:alphabound}{S2}, and the choice \eqref{eqn:wmin} for \hyperlink{hyp:wmin}{S3},
%Assuming \hyperlink{hyp:initial}{S1} - \hyperlink{hyp:wmin}{S3}, 
there exists a constant $C_{11}>0$ so that the following holds true for each $\epsilon,\gamma>0$. If \hyperlink{hyp:Tproperties2}{C5} holds for some $T>0$ and $t\in[0,T]$, then the bound
\begin{align*}\|{v}(t)\|_{L^2_{w_\infty}}\leq\delta_{6}a^{1/2}_\infty,\end{align*} implies
\begin{align}    |\partial_t \|\overline{v}(t)\|^2_{L^2}|
   \leq C_{11} \Big(\epsilon |f(\gamma t)|\|\overline{v}(t)\|_{L^2}^2+  \epsilon|f(\gamma t)| \|v(t)\|_{L^2_{w_\infty}}+\|v(t)\|^2_{L^2_{w_\infty}}\Big).\label{eqn:below}\end{align}
\end{lemma}
\begin{proof}
    Let us write $u=u(t)$, $v=v(t)$, and $\overline{v}=\overline{v}(t)$ for brevity. In the presence of forcing via $\epsilon f(\gamma t) $, we observe that
\begin{align}
    \partial_t \|u\|^2_{L^2} \stackrel{\eqref{eqn:gammaversion}}{=}& -2\langle u, \partial_x^3 u+2 u \partial_x u\rangle +2\epsilon f(\gamma t) \langle u,u\rangle \nonumber\\
    =& 2\epsilon f(\gamma t) \|u\|^2_{L^2}\stackrel{\eqref{eqn:unscaled}}{=}2\epsilon f(\gamma t) \big(\|\phi_{c(t)}\|^2_{L^2}+\|\overline{v}\|^2_{L^2}\big)\nonumber\\
    =&12\epsilon f(\gamma t) c^{3/2}(t)+2\epsilon f(\gamma t)\|\overline{v}\|^2_{L^2},\label{eqn:L2compute}
\end{align}
where we have used that $ \|\phi_{c(t)}\|^2_{L^2} =6 c^{3/2}(t)$ in the last step. On the other hand, we may compute
\begin{align}
    \partial_t \|\phi_{c(t)}\|^2_{L^2} =&6\partial_t c^{3/2}(t)=9 c_t(t)c^{1/2}(t)=-2c_0^{3/2}\alpha^{-4}(t)\alpha_t(t).\label{eqn:solitoncompute}
\end{align}
Combining \eqref{eqn:L2compute} and \eqref{eqn:solitoncompute} and estimating $|\alpha_t(t)+\tfrac{2}{3}\alpha(t) \epsilon f(\gamma t)|$ via \eqref{eqn:modulationestimproved} leads to
the estimate
\begin{align*}
   |\partial_t \|\overline{v}\|^2_{L^2}|\!=\! |\partial_t \|u\|^2_{L^2}\!-\!\partial_t\|\phi_{c(t)}\|^2_{L^2}|
   \!\leq\! \tilde{C}_1 \Big(\epsilon |f(\gamma t)|\|\overline{v}\|_{L^2}^2\!+\!  \epsilon|f(\gamma t)| \|v\|_{L^2_{w_\infty}}\!+\!\|v\|^2_{L^2_{w_\infty}}\Big),
\end{align*}
for some $\tilde{C}_1>0$. 
\end{proof}

%Building on this observation, we establish control of the $L^2$-norm of $\overline{v}$ in the following lemma. Here, we recall that $\delta_6$ is the constant introduced in \Cref{lem:estimates}, and $w_\infty$ is the asymptotic weight introduced in \Cref{lem:weightcondition}.
\begin{lemma}\label{lem:L^2}
Assuming \hyperlink{hyp:initial}{S1}, \hyperlink{hyp:alphabound}{S2}, and the choice \eqref{eqn:wmin} for \hyperlink{hyp:wmin}{S3},
%Assuming \hyperlink{hyp:initial}{S1} - \hyperlink{hyp:wmin}{S3}, 
there exists a constant $C_{12}>0$ so that the following holds true for each $\epsilon,\gamma>0$. If \hyperlink{hyp:exponential}{C4} is satisfied, \hyperlink{hyp:Tproperties2}{C5} holds for some $T>0$, and $\epsilon/\gamma\leq E_{\mathrm{max}}$, then the bound
\begin{align*}e^{\gamma t}\|{v}(t)\|_{L^2_{w_\infty}}\leq\delta_{6}a^{1/2}_\infty , \quad 0\leq t\leq T,\end{align*} implies
\begin{align*} \sup_{0\leq t\leq T}\| \overline{v}(t)\|_{L^2}^2\leq & \|\overline{v}_0\|^2_{L^2}\!+\!C_{12}\Big(\frac{\epsilon}{\gamma}\sup_{0\leq t\leq T}e^{\gamma t}\|{v}(t)\|_{H^1_{w_\infty}}\!+\!\gamma^{-1}\sup_{0\leq t\leq T}e^{2\gamma t}\|{v}(t)\|_{H^1_{w_\infty}}^2\Big).
\end{align*}
\end{lemma}
\begin{proof}
Writing 
\[\epsilon_1=\sup_{0\leq t\leq T}e^{\gamma t}\|{v}\|_{L^2_{w_\infty}},\]
we have via \Cref{lem:L2derivative}
\begin{align*}
    |\partial_t \|\overline{v}\|^2_{L^2}|\leq C_{11} \big(\epsilon e^{-\gamma t}\|\overline{v}\|_{L^2}^2+(\epsilon \epsilon_1  + \epsilon_1^2)e^{-2\gamma t}\big),
\end{align*}
and applying Gr\"{o}nwall's inequality yields 
\begin{equation*}
    \|\overline{v}\|^2_{L^2}\leq \Big(\|\overline{v}_0\|^2_{L^2}+\frac{C_{11}}{2\gamma}(\epsilon\epsilon_1+\epsilon_1^2)\Big)e^{C_{11}\epsilon/\gamma}.    \qedhere
\end{equation*}
\end{proof}
Control in $H^1$ is established using the well-known fact that the soliton $\phi_c$ is a critical point of the functional $\mathcal{E}_{c}[u]=\tfrac{1}{2}c\|u\|_{L^2}^2+\mathcal{H}[u]$. This is reflected in the equality
\begin{align*}
    \mathcal{E}_c[\phi_c+z]-\mathcal{E}_c[\phi_c]=\tfrac{1}{2}c\|z\|_{L^2}^2+\tfrac{1}{2}\|\partial_x z\|_{L^2}^2+\int_{\R}-\phi_c z^2 +\tfrac{1}{3}z^3 \d x,
\end{align*}
which is $O(z^2)$. Substituting our decomposition \eqref{eqn:unscaled} into
\[\mathcal{J}(t)=\mathcal{E}_{c(t)}[u(t)]-\mathcal{E}_{c(t)}[\phi_{c(t)}],\]
we arrive at 
\begin{align}\label{eqn:Jformula}
    \mathcal{J}(t)
    =&\mathcal{E}_{c(t)}[\phi_{c(t)}+\overline{v}(t)]-\mathcal{E}_{c(t)}[\phi_{c(t)}]\nonumber\\
    =&\tfrac{1}{2}c(t)\|\overline{v}(t)\|_{L^2}^2+\tfrac{1}{2}\|\partial_x \overline{v}(t)\|_{L^2}^2+\int_{\R}-\phi_{c(t)} \overline{v}^2(t) +\tfrac{1}{3}\overline{v}^3(t) \d x,
\end{align}
which generalizes \eqref{eqn:pyth} in the sense that the right-hand side is $O(\overline{v}^2)$. In fact, we note that $\|\overline{v}(t)\|_{H^1}^2$ can be bounded in terms of $|\mathcal{J}(t)|$ and vice versa. 
\begin{lemma}
Assuming \hyperlink{hyp:initial}{S1}, \hyperlink{hyp:alphabound}{S2}, and the choice \eqref{eqn:wmin} for \hyperlink{hyp:wmin}{S3},
    %Assuming \hyperlink{hyp:initial}{S1} - \hyperlink{hyp:wmin}{S3}, 
    there exists a constant $C_{13}>0$ so that the following holds true for each $\epsilon,\gamma>0$. If \hyperlink{hyp:Tproperties2}{C5} holds for some $T>0$ and $t\in[0,T]$, then
    \begin{align} \|\overline{v}(t)\|_{H^1}^2\leq
    C_{13}\Big(|\mathcal{J}(t)|+\|{v}(t)\|_{L_{w_\infty}^2}^2+\|\overline{v}(t)\|_{H^1}\|\overline{v}(t)\|_{L^2}^2\Big),\label{eqn:Jest}\end{align}
    and
    \begin{align}|\mathcal{J}(t)|\leq  C_{13}\Big(\|\overline{v}(t)\|_{H^1}^2+\|{v}(t)\|_{L_{w_\infty}^2}^2 +\|\overline{v}(t)\|_{H^1}\|\overline{v}(t)\|_{L^2}^2\Big)\label{eqn:forJ0}.\end{align}
\end{lemma}
% On the one hand, it is straightforward to see that
% \begin{align*}
%     |\mathcal{J}(t)|\leq (\tfrac{1}{2}c(t)+\tfrac{1}{2}+\|\phi_{c(t)}\|_{L^\infty}+\tfrac{\sqrt{2}}{3}\|\overline{v}(t)\|_{H^1})\|\overline{v}(t)\|_{H^1}^2.
% \end{align*}
\begin{proof}
Estimating 
\begin{align}\left|\int_{\R}\phi_{c(t)} \overline{v}^2(t)\d x\right|=\left|\alpha^{-5}(t)\int_{\R}\phi_{c_0} v^2(t)\d x\right|\leq & \alpha^{-5}_{\mathrm{min}}\|e^{-2w_\infty x}\phi_{c_0}\|_{L^\infty}\|{v}(t)\|_{L_{w_\infty}^2}^2 \label{eqn:forJ}
\end{align}
in \eqref{eqn:Jformula} yields 
\begin{align*}
 |\mathcal{J}(t)|\leq &\tfrac{1}{2} \max\{c_0 \alpha_{\mathrm{min}}^{-2},1\}\|\overline{v}(t)\|_{H^1}^2+\alpha^{-5}_{\mathrm{min}}\|e^{-2w_\infty x}\phi_{c_0}\|_{L^\infty}\|{v}(t)\|_{L_{w_\infty}^2}^2 \\
 &+\tfrac{\sqrt{2}}{3}\|\overline{v}(t)\|_{H^1}\|\overline{v}(t)\|_{L^2}^2,
\end{align*}
which yields \eqref{eqn:forJ0}. On the other hand, rewriting \eqref{eqn:Jformula} as
\begin{align*}
     \tfrac{1}{2}c(t)\|\overline{v}(t)\|_{L^2}^2+\tfrac{1}{2}\|\partial_x \overline{v}(t)\|_{L^2}^2
    =&-\mathcal{J}(t)+\int_{\R}\phi_{c(t)} \overline{v}^2(t)-\tfrac{1}{3}\overline{v}^3(t) \d x,
\end{align*}
and applying \eqref{eqn:forJ} provides
\begin{align*}
    k\|\overline{v}(t)\|_{H^1}^2
    \leq &|\mathcal{J}(t)|+\alpha^{-5}_{\mathrm{min}}\|e^{-2w_\infty x}\phi_{c_0}\|_{L^\infty}\|{v}(t)\|_{L_{w_\infty}^2}^2+\tfrac{\sqrt{2}}{3}\|\overline{v}(t)\|_{H^1}\|\overline{v}(t)\|_{L^2}^2,
\end{align*}
in case $k\leq \min\{\tfrac{1}{2},\tfrac{1}{2}c_0\alpha^{-2}_{\mathrm{min}}\}$, which establishes \eqref{eqn:Jest}.
\end{proof}
We now establish an estimate of $\partial_t \mathcal{J}(t)$, which we use to control $\|\overline{v}\|_{H^1}$ via a Gr\"{o}nwall argument, analogous to the proof of \Cref{lem:L^2}. Note also in \eqref{eqn:below2} below, that the term $\|\overline{v}(t)\|_{H^1}^2$ carries a factor $\epsilon|f(\gamma t)|$.
\begin{lemma}\label{lem:Jderivative}
Assuming \hyperlink{hyp:initial}{S1}, \hyperlink{hyp:alphabound}{S2}, and the choice \eqref{eqn:wmin} for \hyperlink{hyp:wmin}{S3}, there exists a constant $C_{14}>0$ so that the following holds true for each $\epsilon,\gamma>0$. If \hyperlink{hyp:Tproperties2}{C5} holds for some $T>0$ and $t\in[0,T]$, then the bounds
\begin{align*}\|{v}(t)\|_{L^2_{w_\infty}}\leq \min\{1,\delta_{6}a^{1/2}_\infty\} \quad \text{and} \quad \|\overline{v}(t)\|_{L^2}\leq  1,\end{align*} 
imply
\begin{align}|\partial_t \mathcal{J}(t)|\leq C_{14}\Big(\epsilon |f(\gamma t)|\| \overline{v}(t)\|_{H^1}^2+\epsilon |f(\gamma t)|\|v(t)\|_{L^2_{w_\infty}}+ C_{14}\|v(t)\|_{L^2_{w_\infty}}^2\Big).\label{eqn:below2}\end{align}
\end{lemma}
\begin{proof}
    We once more abbreviate $u=u(t)$, $v=v(t)$, and $\overline{v}=\overline{v}(t)$, and compute that 
\begin{align*}
    \partial_t \mathcal{H}[u]=&\partial_t\int_{\R}\tfrac{1}{2}(u_x)^2-\tfrac{1}{3}u^3 \d x=\int_{\R}u_x u_{xt}-u^2u_t \d x=-\int_{\R}(u_{xx} +u^2)u_t \d x\\
    \stackrel{\eqref{eqn:gammaversion}}{=}&\big\langle u_{xx} +u^2,u_{xxx}+(u^2)_x-\epsilon f(\gamma t)  u\big\rangle =\epsilon f(\gamma t) \|\partial_x u\|^2 -\epsilon f(\gamma t)  \int_{\R} u^3 \d x.
\end{align*}
Substituting \eqref{eqn:unscaled}, we have
\begin{align*}
    \partial_t \mathcal{H}[u]=&\epsilon f(\gamma t) \|\partial_x u\|_{L^2}^2 -\epsilon f(\gamma t)  \int u^3 \d x = \epsilon f(\gamma t) \big(\|\partial_x \phi_{c(t)} \|_{L^2}^2-\int_\R \phi^3_{c(t)} \d x \big)\\
    &+\epsilon f(\gamma t) \big(2\langle \partial_x\phi_{c(t)},\partial_x \overline{v}\rangle-3\langle \phi^2_{c(t)},\overline{v}\rangle \big)\\
    &+\epsilon f(\gamma t)\big(\|\partial_x \overline{v}\|_{L^2}^2-\int_{\R}3\phi_{c(t)}\overline{v}^2+\overline{v}^3 \d x\big),
\end{align*}
where the leading-order term evaluates to
\[\|\partial_x \phi_{c(t)} \|_{L^2}^2-\int_\R \phi^3_{c(t)} \d x=-6c^{5/2}(t).\]
On the other hand, 
\begin{align*}
    \partial_t \mathcal{H}[\phi_{c(t)}]=-\tfrac{9}{5}\partial_t c^{5/2}(t)=-\tfrac{9}{2} c_t(t)c^{3/2}(t)=9c_0^{5/2}\alpha^{-6}(t)\alpha_t(t)
\end{align*}
and we note via \eqref{eqn:modulation} that $\partial_t \mathcal{H}[u]=\partial_t \mathcal{H}[\phi_{c(t)}]$ in case $v=0$. Moreover,  estimating \[|\alpha_t(t)+\tfrac{2}{3}\alpha(t) \epsilon f(\gamma t)|\leq \tilde{C}_1 \big( \epsilon|f(\gamma t)| \|v\|_{L^2_{w_\infty}}+\|v\|^2_{L^2_{w_\infty}}\big)\] 
for some $\tilde{C}_1>0$ leads to
\begin{align*}
    | \partial_t \mathcal{H}[u]-\partial_t \mathcal{H}[\phi_{c(t)}]|\leq &
    \epsilon |f(\gamma t)| \big(2\langle \partial_x\phi_{c(t)},\partial_x \overline{v}\rangle-3\langle \phi^2_{c(t)},\overline{v}\rangle \big]\\
    &+\epsilon |f(\gamma t)|\big(\|\partial_x \overline{v}\|_{L^2}^2-\int_{\R}3\phi_{c(t)}\overline{v}^2+\overline{v}^3 \d x\big)\\
    &+\tilde{C}_1 \big( \epsilon|f(\gamma t)| \|v\|_{L^2_{w_\infty}}+\|v\|^2_{L^2_{w_\infty}}\big).
\end{align*}
Rescaling the terms that contain the soliton and introducing an exponential weight yields
\begin{align*}|2\langle \partial_x\phi_{c(t)},\partial_x \overline{v}\rangle-3\langle \phi^2_{c(t)},\overline{v}\rangle|\leq & \tilde{C}_2|2\langle \partial_x\phi_{c_0},\partial_x {v}\rangle-3\langle \phi^2_{c_0},{v}\rangle|\\
\leq & \tilde{C}_2\|\partial_x^2\phi_{c_0}+\phi_{c_0}^2\|_{L^2_{-w_\infty}}\|v\|_{L^2_{w_\infty}} \end{align*}
and
\begin{align*}
    \left|\int_{\R}3\phi_{c(t)}\overline{v}^2 \d x\right|\leq \tilde{C}_2\left|\int_{\R}3\phi_{c_0}{v}^2 \d x\right|\leq \tilde{C}_2\|e^{-2w_\infty x}\phi_{c_0}\|_{L^\infty}\|v\|_{L_{w_\infty}^2}^2
\end{align*}
for some constant $\tilde{C}_2>0$. The cubic term is controlled by estimating
\begin{align*}
\left|\int_{\R}\overline{v}^3 \d x\right|\leq \|\overline{v}\|_{L^\infty}\|\overline{v}\|_{L^2}^2\leq \sqrt{2}\|\overline{v}\|_{H^1}^2,
\end{align*}
where we have used %that
$\|\overline{v}\|_{L^2}\leq 1$.
% \begin{align*}
%     \tilde{C}_2 \epsilon |f(\gamma t)|[\|\partial_x^2\phi_{c_0}+\phi_{c_0}^2\|_{L^2_{-w_\infty}}\|v(t)\|_{L^2_{w_\infty}}+\|e^{-2w_\infty x}\phi_{c_0}\|_{L^\infty}\|v(t)\|_{L_{w_\infty}^2}^2]\\
%     &+\tilde{C}_2 \epsilon |f(\gamma t)|[1+\|\overline{v}(t)\|_{L^2}]\| \overline{v}(t)\|_{H^1}^2\\
%     &+\tilde{C}_1 \big[ \epsilon|f(\gamma t)| \|v(t)\|_{L^2_{w_\infty}}+\|v(t)\|^2_{L^2_{w_\infty}}\big],
% \end{align*}
Collecting our results so far yields
\begin{align*}
    \big| \partial_t \mathcal{H}[u]-\partial_t \mathcal{H}[\phi_{c(t)}]\big|\leq &
    (1+\sqrt{2})\epsilon |f(\gamma t)|\| \overline{v}\|_{H^1}^2+\tilde{C}_3 \epsilon |f(\gamma t)|\big(\|v\|_{L^2_{w_\infty}}+\|v\|_{L^2_{w_\infty}}^2\big)\\
    &+\tilde{C}_3\|v\|_{L^2_{w_\infty}}^2
\end{align*}
for some $\tilde{C}_3>0$. Computing
\[\partial_t\mathcal{J}(t)=c_t(t)\|\overline{v}\|_{L^2}^2+c(t)\partial_t\|\overline{v}\|_{L^2}^2+ \partial_t \mathcal{H}[u]-\partial_t \mathcal{H}[\phi_{c(t)}],\]
applying \Cref{lem:L2derivative}, and estimating $|c_t|$ via \Cref{lem:estimates} then gives the result.
% \begin{align*}
%     |\partial_t \mathcal{J}(t)|\leq (1+\sqrt{2})\epsilon e^{-\gamma t}\| \overline{v}\|_{H^1}^2+\tilde{C}_4 e^{-\gamma t}\big(\epsilon \epsilon_1(1+\epsilon_2+\epsilon_1+\epsilon\epsilon_1)+\epsilon\epsilon_2+\epsilon_1^2\big)
% \end{align*}
% for some $\tilde{C}_4>0$.
\end{proof}
We are now ready to state and prove the main result of this section, which establishes control over the $H^1$-norm of the perturbation.

\begin{proposition}\label{prop:H1}
Assuming \hyperlink{hyp:initial}{S1}, \hyperlink{hyp:alphabound}{S2}, and the choice \eqref{eqn:wmin} for \hyperlink{hyp:wmin}{S3},
%Assuming \hyperlink{hyp:initial}{S1} - \hyperlink{hyp:wmin}{S3}, 
there exists a constant $ C_{15}>0$ so that the following holds true for each $\epsilon,\gamma>0$. If \hyperlink{hyp:exponential}{C4} is satisfied, \hyperlink{hyp:Tproperties2}{C5} holds for some $T>0$, and $\epsilon/\gamma\leq E_{\mathrm{max}}$, then the bounds
\begin{align*}\sup_{0\leq t\leq T}e^{\gamma t}\|{v}(t)\|_{L^2_{w_\infty}}\leq &\min\{\delta_{6}w_\infty^{1/2},1\},\\
\sup_{0\leq t\leq T}\|\overline{v}(t)\|_{H^1}\leq & 1,
\end{align*} imply
\begin{align*} \sup_{0\leq t \leq T}\| {v}(t)\|_{H^1}^2\leq & C_{15}(\|\overline{v}_*\|^2_{H^1}+\|\overline{v}_*\|^2_{H^1_{w}})\\
&+C_{15}\Big(\frac{\epsilon}{\gamma}\sup_{0\leq t\leq T}e^{\gamma t}\|{v}(t)\|_{H^1_{w_\infty}}+\gamma^{-1}\sup_{0\leq t\leq T}e^{2\gamma t}\|{v}(t)\|_{H^1_{w_\infty}}^2\Big).
\end{align*}
\end{proposition}

\begin{proof}
We again write
\[\epsilon_1=\sup_{0\leq t\leq T}e^{\gamma t}\|{v}\|_{L^2_{w_\infty}} \quad \text{and} \quad \epsilon_2=\sup_{0\leq t\leq T}\|\overline{v}\|_{L^2}^2.\]
Applying \eqref{eqn:Jest}, we obtain
\begin{align*}
    \|\overline{v}\|_{H^1}^2
    \leq &C_{13}|\mathcal{J}(t)|+C_{13}\|{v}\|_{L_{w_\infty}^2}^2+C_{13}\|\overline{v}\|_{H^1}\|\overline{v}\|_{L^2}^2\\
    \leq & C_{13}|\mathcal{J}(0)|+C_{13}\int_0^t|\partial_t \mathcal{J}(s)|\d s+C_{13}(\epsilon_1^2+\epsilon_2),
 \end{align*}
which, using \Cref{lem:Jderivative}, leads to the bound
\begin{align*}
   \|\overline{v}\|_{H^1}^2 \leq & C_{13}|\mathcal{J}(0)|+C_{13}C_{14}\int_0^t\epsilon e^{-\gamma s}\| \overline{v}(s)\|_{H^1}^2\d s+C_{13}(\epsilon_1^2+\epsilon_2)\\
    &+C_{13}C_{14} \big(\frac{\epsilon}{\gamma}(\epsilon_1+\epsilon_2)+\frac{\epsilon_1^2}{\gamma}\big).
\end{align*}
% \begin{align*}
%     k\|\overline{v}\|_{H^1}^2\leq &|\mathcal{J}(t)|+\|e^{-2w_\infty x}\phi_{c(t)}\|_{L^\infty}\|\overline{v}\|_{L_{w_\infty}^2}^2+\tfrac{\sqrt{2}}{3}\|\overline{v}\|_{H^1}\|\overline{v}\|_{L^2}^2\\
%     \leq & |\tfrac{1}{2}c(t)\|\overline{v}\|_{L^2}^2+ \mathcal{H}[u]-\mathcal{H}[\phi_{c(t)}]|+\tilde{C}_4\epsilon_1^2+\tfrac{\sqrt{2}}{3}\|\overline{v}\|_{L^2}^2\\
% \leq & (\tfrac{1}{2}c_{\mathrm{max}}+\tfrac{\sqrt{2}}{3})B_{L^2}+\tilde{C}_4 \epsilon_1^2+ |\mathcal{H}[\phi_{c_0}+v_0]-\mathcal{H}[\phi_{c_0}]|\\
% &+\int_0^t| \partial_t \mathcal{H}[u(s)]-\partial_t \mathcal{H}[\phi_{c(s)}]|\d s\\
% \leq & (\tfrac{1}{2}c_{\mathrm{max}}+\tfrac{\sqrt{2}}{3})B_{L^2}+\tilde{C}_4 \epsilon_1^2+ |\mathcal{H}[\phi_{c_0}+v_0]-\mathcal{H}[\phi_{c_0}]|\\
% &+\tilde{C}_2\epsilon\int_0^t e^{-\gamma s}\| \overline{v}(s)\|_{H^1}^2\d s+\tilde{C}_3\frac{\epsilon}{\gamma}[\epsilon_1+(1+\epsilon)\epsilon_1^2].
% \end{align*}
Applying Gr\"{o}nwall's inequality finally yields
\begin{align*}
    \|\overline{v}\|_{H^1}^2\leq &\Big(C_{13}|\mathcal{J}(0)|+C_{13}(\epsilon_1^2+\epsilon_2)+C_{13}C_{14} \big(\frac{\epsilon}{\gamma}(\epsilon_1+\epsilon_2)+\frac{\epsilon_1^2}{\gamma}\big)\Big)e^{C_{13}C_{14}\frac{\epsilon}{\gamma}}.\end{align*}
The result now follows by controlling $\epsilon_2$ via \Cref{lem:L^2}, controlling $|\mathcal{J}(0)|$ via \eqref{eqn:forJ0}, and applying item 1 of \Cref{lem:scaling}.
% \begin{align*}
%      \mathcal{J}(0)=&\mathcal{E}_{c(0)}[u(0)]-\mathcal{E}_{c(0)}[\phi_{c(0)}]=\mathcal{E}_{c(0)}[\phi_{c(0)}+\overline{v}(0,\cdot+\xi_0)]-\mathcal{E}_{c(0)}[\phi_{c(0)}]\\
%      =&\tfrac{1}{2}c\|\overline{v}_0\|_{L^2}^2+\tfrac{1}{2}\|\partial_x \overline{v}_0\|_{L^2}^2+\int_{\R}-\phi_c \overline{v}^2(0) +\tfrac{1}{3}\overline{v}^3(0) \d x.
% \end{align*}
\end{proof}

\section{Proof of main result}\label{sec:proofmain}
% \begin{Lemma}\label{lem:alphabounds}
% Assume \hyperlink{hyp:exponential}{C4} and let $\epsilon_1>0$ be a constant for which
%     \begin{align*}\|{v}(t)\|_{H^1_{w_\infty}}\leq\epsilon_1 e^{-b\epsilon t}, \quad 0\leq t\leq T.\end{align*} Then
% \begin{align*}\left|\log({\alpha}(t))+ \tfrac{2}{3}\int_0^{\epsilon t}f(s) \d s\right|\lesssim\epsilon_1+\tfrac{\epsilon_1^2}{\epsilon}, \quad 0\leq t\leq T,
% \end{align*}
% and in particular there exist constants  $0<\alpha_{\mathrm{min}}<1 < \alpha_{\mathrm{max}}$ independent of $T$ such that
% \[\alpha(t)\in [\alpha_{\mathrm{min}},\alpha_{\mathrm{max}}], \quad t\in[0,T].\]
% \end{Lemma}

Building upon the results of the previous sections, we set out to prove \Cref{thm:bigthm}. Assuming \hyperlink{hyp:initial}{S1}, we fix the constants appearing in \hyperlink{hyp:alphabound}{S2} by writing
\begin{equation}
    \label{eq:defn:alpha:pm}
    \alpha_{\mathrm{min}}=\tfrac{1}{2}\inf_{t\geq 0}e^{\frac{2}{3}E_{\mathrm{max}}\int_0^{ t}f(s) \d s}\quad \text{and} \quad \alpha_{\mathrm{max}}=2\sup_{t\geq 0}e^{\frac{2}{3}E_{\mathrm{max}}\int_0^{t}f(s) \d s}.
\end{equation}
%based only on \hyperlink{hyp:initial}{S1}. 
Recall that the weight $w_{\mathrm{min}}$ of \hyperlink{hyp:wmin}{S3} and the asymptotic weight $w_\infty$ are then determined through \eqref{eqn:wmin}, in which $\delta_6$ and $C_6$ are the constants introduced in \Cref{lem:estimates}, which depend only on \hyperlink{hyp:initial}{S1} and \hyperlink{hyp:alphabound}{S2}. As a final preparation,  %for the proof of \Cref{thm:bigthm}, 
we estimate the deviation of the  modulation parameters from their leading-order approximations. Recall that ${\xi}(t)-\xi_0-\int_0^t {c}(s)\ \d s={\Omega}(t)$. 

\begin{lemma}\label{lem:approx}
Assume \hyperlink{hyp:initial}{S1} and the choices \eqref{eq:defn:alpha:pm} and \eqref{eqn:wmin} for \hyperlink{hyp:alphabound}{S2} and \hyperlink{hyp:wmin}{S3}.
%and the choices above for \hyperlink{hyp:alphabound}{S2} and \hyperlink{hyp:wmin}{S3}, 
If \hyperlink{hyp:exponential}{C4} is satisfied and \hyperlink{hyp:Tproperties2}{C5} holds for some $T>0$, then for each $\epsilon, \gamma>0$, the bound 
\begin{align*}e^{\gamma t}\|{v}(t)\|_{L^2_{w_\infty}}\leq\delta_{6}w_\infty^{1/2} , \quad 0\leq t\leq T\end{align*} 
implies
\begin{align*}
   &\sup_{t\in[0,T]} \left|\log({\alpha}(t))+\tfrac{2}{3}\epsilon\int_0^{ t}f(\gamma s) \d s\right|+\sup_{t\in[0,T]}\left|\Omega(t)-\tfrac{2}{3}\epsilon\int_0^t \frac{ f(\gamma s)}{c^{1/2}(s)} \d s\right|\\
   &\leq C_6 w_{\infty}^{-1/2}\frac{\epsilon}{\gamma}\sup_{t\in[0,T]}e^{\gamma t}\|v(t)\|_{L^2_{w_\infty}}
   +\frac{C_6}{\gamma}\sup_{t\in[0,T]}e^{2\gamma t}\|v(t)\|^2_{L^2_{w_\infty}}.
\end{align*}
\end{lemma}
\begin{proof}
Estimating \eqref{eqn:modulation} for $t\in[0,T]$ as in \eqref{eqn:modulationestimproved} yields
\begin{align*}
    \left|\partial_t\log({\alpha}(t))+\tfrac{2}{3}\epsilon f(\gamma t) \right|+\left|\partial_t\Omega(t)-\tfrac{2}{3}\epsilon \frac{f(\gamma t)}{c^{1/2}(t)}\right|
    \leq & C_6w_{\infty}^{-1/2}\epsilon e^{-\gamma t} \|v(t)\|_{L^2_{w_\infty}}\\
    &+C_6\|v(t)\|^2_{L^2_{w_\infty}}.
\end{align*}
The result now follows by integrating.
% \begin{align*}
% \sup_{0\leq t\leq  T_{\mathrm{max}}}\left|\log({\alpha}(t))+\tfrac{2}{3}\epsilon \int_0^{{ t}}f(\gamma s) \d s\right|
%     \leq & \tilde{C}_1\frac{\epsilon}{\gamma}(\|\overline{v}_*\|_{H^1_{w}}+\|\overline{v}_*\|_{H^1}+ \epsilon^{1-2p})\\
%     &+\tilde{C}_1 \gamma^{-1}(\|\overline{v}_*\|^2_{H^1_{w}}+\|\overline{v}_*\|^2_{H^1}+\epsilon^{2-4p})
% \end{align*}
\end{proof}
For each $\epsilon,E>0$ and $\overline{v}_*\in H^2\cap H^1_w$, we now introduce the sets
\begin{align*}O_1=&\big\{T\geq 0 : e^{\gamma t}\|{v}(t)\|_{H^1_{\mathbf{w}(t)}}\leq C_8(\|\overline{v}_*\|_{H^1_{w}}+\|\overline{v}_*\|_{H^1}+ \epsilon^{1-2p}) \hbox{ for all } 0\leq t\leq T\big\},\\
O_2=&\big\{T\geq 0 : \|{v}(t)\|^2_{H^1}\leq \alpha_{\mathrm{min}}^{5} 
\hbox{ for all } 0\leq t\leq T\big\},\end{align*}
where $C_8$ is the constant from \Cref{prop:duhamel}, and we recall that the weight-function $\mathbf{w}(t) = \big(w_-(t), w_+(t) \big)$ is defined in \eqref{eq:defn:w:pm}. % and \eqref{eqn:weightdef}. 
We then define $ T_{\mathrm{max}}(\epsilon,\gamma,\overline{v}_*)\in [0,\infty]$ as %through
\[T_{\mathrm{max}}(\epsilon,\gamma,\overline{v}_*)=\sup (O_1 \cap O_2).\]
Recalling that $E=\epsilon/\gamma$, the key ingredient toward establishing \Cref{thm:bigthm} is to show that $ T_{\mathrm{max}}(\epsilon,\gamma,\overline{v}_*)=\infty$. 

\begin{proof}[Proof of \Cref{thm:bigthm}]
Pick $\epsilon\in (0,1]$ and $\gamma>0$ that satisfies $ \epsilon/\gamma\in (0, E_{\mathrm{max}}]$ and
\[ \epsilon^{1+p}+\epsilon^{2-4p}\leq  \delta_1 \gamma, \quad \epsilon^{p}\geq \gamma , \quad \gamma\leq \delta_1,\]
and consider an initial condition $\overline{v}_*\in H^2\cap H^1_w$ 
that satisfies $
\|\overline{v}_*\|^2_{H^1}+\|\overline{v}_*\|^2_{H^1_w}\leq \delta_1\gamma$.
Then, the conditions in \Cref{thm:bigthm} are met, and we note that \[\epsilon \leq \delta_1 E_{\mathrm{max}}\quad  \text{and}\quad  \|\overline{v}_*\|_{H^1}+\|\overline{v}_*\|_{H^1_w}\leq \sqrt{2}\delta_1.\]
Suppose, for the sake of contradiction, that $ T_{\mathrm{max}}<\infty$. By construction of $O_1$, we have
\[\sup_{t\in[0,T_{\mathrm{max}})}e^{\gamma t}\|{v}(t)\|_{H^1_{\mathbf{w}(t)}}\leq C_8\big(\|\overline{v}_*\|_{H^1_{w}}+\|\overline{v}_*\|_{H^1}+ \epsilon^{1-2p}\big)\leq  C_8(\sqrt{2}+1)\delta_1,\]
and
\[\sup_{t\in[0,T_{\mathrm{max}})}\gamma^{-1}e^{2\gamma t}\|{v}(t)\|^2_{H^1_{\mathbf{w}(t)}}\leq 3C_8^2\gamma^{-1}\big(\|\overline{v}_*\|_{H^1_{w}}^2+\|\overline{v}_*\|_{H^1}^2+ \epsilon^{2-4p}\big)\leq 6C_8^2\delta_1.\]
In particular, \[\sup_{t\in[0,T_{\mathrm{max}})}e^{\gamma t}\|{v}(t)\|_{H^1_{\mathbf{w}(t)}}\leq  \min\{\delta_{6}w_\infty^{1/2},1\}\] for $\delta_1$ sufficiently small. Thus, we may use \Cref{lem:approx} to obtain 
\begin{align*}
   \sup_{t\in[0,T_{\mathrm{max}})} \left|\log({\alpha}(t))+\tfrac{2}{3}\epsilon\int_0^{ t}f(\gamma s) \d s\right|
   \leq & C_6 C_8w_{\infty}^{-1/2} \frac{\epsilon}{\gamma}\big(\|\overline{v}_*\|_{H^1_{w}}+\|\overline{v}_*\|_{H^1}+ \epsilon^{1-2p}\big)
   \\
   &+C_6 C_8^2\gamma ^{-1}\big(\|\overline{v}_*\|_{H^1_{w}}+\|\overline{v}_*\|_{H^1}+ \epsilon^{1-2p}\big)^2\\
   \leq & C_6C_8w_{\infty}^{-1/2}(\sqrt{2}E_{\mathrm{max}}\delta_1 +\delta_1)+6C_6C_8^2\delta_1
\end{align*}
and in particular
\[\alpha_{\mathrm{min}}<\alpha(t)<\alpha_{\mathrm{max}} ,\quad t\in [0, T_{\mathrm{max}}],\]
for $\delta_1$ sufficiently small. By construction of $O_2$ and \Cref{lem:scaling}, we have
\[\sup_{t\in[0,T_{\mathrm{max}})}\|\overline{v}(t)\|_{H^1}^2\leq 1.\]
Via \Cref{prop:H1}, we may improve this bound to 
\begin{align*}\sup_{t\in[0,T_{\mathrm{max}})}\|{v}(t)\|_{H^1}^2\leq & C_{15}(\|\overline{v}_*\|^2_{H^1}\!+\!\|\overline{v}_*\|^2_{H^1_{w_\infty}})
\!+\!C_{15}C_8\frac{\epsilon}{\gamma}\big(\|\overline{v}_*\|_{H^1_{w}}\!+\!\|\overline{v}_*\|_{H^1}+ \epsilon^{1-2p}\big)\\
&+C_{15}C_8^2\gamma^{-1}\big(\|\overline{v}_*\|_{H^1_{w}}+\|\overline{v}_*\|_{H^1}+ \epsilon^{1-2p}\big)^2\\
\leq &C_{15}\delta_1^2+C_{15}C_8(\sqrt{2}E_{\mathrm{max}}\delta_1 +\delta_1)+6C_{15}C_8^2\delta_1.\end{align*}
Combining our results so far, we find via \eqref{eqn:kdef} that
\begin{align*}
    K_{\epsilon,\gamma}(T_{\mathrm{max}}-q)\leq& \sup_{s\in[0,T_{\mathrm{max}})}\big(\|{v}
(s)\|_{H^1}+e^{ \gamma s}\|{v}
(s)\|_{H^1_{\mathbf{w}(s)}}+\gamma^{-1}e^{2 \gamma s}\|{v}
(s)\|^2_{H^1_{\mathbf{w}(s)}}\big)\\
&+E_{\mathrm{max}}\delta_1+\delta_1
<\delta_8
\end{align*}
 for any $q\in (0,T_{\mathrm{max}}]$, decreasing the size of $\delta_1 > 0$ if necessary. By continuity of 
 \[t \mapsto e^{\gamma t}\| v(t)\|_{H^1_{\mathbf{w}(t)}} \quad  \text{and} \quad t \mapsto \| v(t)\|_{H^1},\] there must be a small $r>0$ for which
% \[\sup_{t\in[0,T_{\mathrm{max}}+r]}\big[e^{\gamma t}\| v(t)\|_{H^1_{\mathbf{w}(t)}}+\| v(t)\|_{H^1}\big]\leq 2/3 \delta_{8} \]
\begin{align*}
    K_{\epsilon,\gamma}(T_{\mathrm{max}}-q+r)\leq \delta_{8}.
\end{align*}
Having established a priori control on $K_{\epsilon,\gamma}$, we may apply \Cref{prop:duhamel} to obtain
\[\sup_{t\in[0,T_{\mathrm{max}}-q+r]}e^{\gamma t}\| v(t)\|_{H^1_{\mathbf{w}(t)}}\leq C_8\big(\|\overline{v}_*\|_{H^1_{w}}+\|\overline{v}_*\|_{H^1}+ \epsilon^{1-2p}\big).\]
Choosing $q<r$ shows that $ T_{\mathrm{max}}$ is not maximal, allowing us to conclude that, indeed, $ T_{\mathrm{max}}=\infty$. 

To complete the proof, we now observe that
\begin{align*}\sup_{t \geq 0}e^{\gamma t}\|{v}(t)\|_{H^1_{w_\infty}}\leq \sup_{t \geq 0}e^{\gamma t}\| v(t)\|_{H^1_{\mathbf{w}(t)}} \leq C_8\big(\|\overline{v}_*\|_{H^1_{w}}+\|\overline{v}_*\|_{H^1}+ \epsilon^{1-2p}\big),
\end{align*}
which establishes item 2. Items 3 and 4 follow by applying \Cref{lem:approx}, while item 1 follows from \Cref{prop:H1}.
%with $T=\infty$.
% and 
% \[\sup_{t \geq 0}\|{v}(t)\|^2_{H^1}\leq C_{15}(\|\overline{v}_*\|^2_{H^1}+\|\overline{v}_*\|^2_{H^1_{w_\infty}})
% +6C_{15}C_8^2 \epsilon^{1-4p}.\]
\end{proof}

\appendix

\section{Decompositions} %Classical stability results}
\label{app:old}

 Our work makes frequent use of the linear stability theory of the operator $\mathcal{L}_c$ developed in \cite{pegoweinstein}. Let us briefly review these results here, based on the exposition in \cite{mizumachi}. Introducing the exponential weight $e^{wx}$ on $L^2$ moves the essential spectrum of the operator $\mathcal{L}_c$ from the imaginary axis into the stable halfplane, leaving a double eigenvalue at 0. This 0-eigenvalue has a two-dimensional generalized kernel spanned by $\partial_x\phi_{c}$ and $\partial_c \phi_{c}$. In particular, it is easily verified that $\mathcal{L}_{c} \partial_x\phi_{c}=0$ and $\mathcal{L}_{c} \partial_c \phi_{c}=\partial_x\phi_{c}$. The operator $\mathcal{L}_c$ is related to its (formal) adjoint $\mathcal{L}^*_c$ on the space $L^2_{-w}$ via the relation $\partial_x\mathcal{L}_c^*=-\mathcal{L}_c\partial_x$. In particular, $\mathcal{L}^*_c$ has a two-dimensional kernel spanned by $\phi_c$ and the primitive
\[\zeta_c (x)=\int_{-\infty}^x \partial_c \phi_c(y)\d y,\]
which satisfies $\zeta_c \in L^2_{-w}$.
The spectral projection onto the generalized kernel of $\mathcal{L}_c$ is given by
\begin{align}P_{c} f=\langle f, \eta_c^1 \rangle \partial_x \phi_{c} +\langle f, \eta_c^2\rangle \partial_c \phi_c,\label{eqn:spectral}\end{align}
where 
\begin{align}\eta_c^1=\tfrac{2}{9}c^{-1/2}\zeta_c+\tfrac{2}{9}c^{-2}\phi_c \quad \text{and} \quad \eta_c^2=\tfrac{2}{9}c^{-1/2}\phi_c, \label{eqn:linearcomb}\end{align}
and we write  $Q_c=I-P_c$ for the complementary projection. Based on the spectral properties of $\mathcal{L}_c$, the following can be concluded about the flow generated by this operator.

\begin{theorem}[\cite{pegoweinstein}]\label{thm:linearstability}
    Let $c>0$ and $0<w<\sqrt{c}$. Then, $\mathcal{L}_{c}$ is the generator of a $C_0$-semigroup on $H^s_w$ for any real $s$. For all $\beta>0$ which satisfy $\beta<w(c-w^2)$, there exists a constant $C>0$ such that, for all $g\in L_{w}^2$, $t>0$ and $k\in\{0,1\}$, we have
    \begin{align}\label{eqn:linear}
        \|\partial_x^k  e^{\mathcal{L}_{c}t}Q_c g\|_{L^2_{w}}\leq Ct^{-k/2}e^{-\beta t}\|g\|_{L^2_{w}}.
    \end{align}
\end{theorem}

We conclude here with a result regarding the orthogonality conditions arising from the projection $P_c$. This result ensures that
the decomposition \eqref{eqn:unscaled} underlying the arguments
in this paper is uniquely defined.

\begin{lemma}[\cite{mizumachi}]\label{lem:implicit}
    Let $c_*>0$ and $w\in(0,\sqrt{c_*})$. Then, there exist constants $\delta_2, C_2>0$ such that,  for each $\overline{v}_* \in H^1_w$, the bound $\|\overline{v}_*\|_{L^2_w}\leq \delta_2$ implies that there exist unique parameters $c_0>0$ and $\xi_0\in \R$ such that
    \[\phi_{c_*}(x)+\overline{v}_*(x)=\phi_{c_0}(x-\xi_0)+\overline{v}_0(x-\xi_0) \quad \text{with} \quad \langle\overline{v}_0,\phi_{c_0}\rangle =\langle\overline{v_0},\zeta_{c_0}\rangle=0\]
    and
    \begin{align*}\|\overline{v}_0\|_{H^1_w}+|\xi_0|+|c_*-c_0|\leq &C_2\|\overline{v}_*\|_{H^1_w},\\
    \|\overline{v}_0\|_{H^1}\leq& C_2(\|\overline{v}_*\|_{H^1}+\|\overline{v}_*\|_{H^1_w}).
    \end{align*}
\end{lemma}

\section{Time-varying weights}\label{app:weight}

Our goal here is to establish several properties of the time-dependent weights $\mathbf{w}(t)=\big(w_-(t),w_+(t)\big)$ given by
\[w_-(t)=w_{\mathrm{min}}(\tfrac{w_\infty}{w_{\mathrm{min}}})^{1-e^{- \gamma  t}}, \quad w_+(t)=w(\tfrac{w_\infty}{w})^{1-e^{- \gamma  t}}\]
that were used to control the perturbation over long time-scales in \Cref{sec:longtime}. In particular, we prove \Cref{lem:weightcondition}.

\begin{proof}[Proof of \Cref{lem:weightcondition}] 
Writing
\[\epsilon_1=\sup_{t\in [0,T]}e^{\gamma t}\|{v}
(t)\|_{L^2_{w_\infty}},\]
we note that for $s\in[0,\delta]$, we may compute % it holds that
\begin{align*}
    \big|\log\big({\alpha}(t)\big)-\log\big({\alpha}(t+s)\big)\big|= &\left|\int_t^{t+s}\log({\alpha}(r))^\prime \d r\right| \\ \stackrel{\eqref{eqn:modulationestimate}}{\leq} &C_6\int_t^{t+s}\epsilon e^{-\gamma r} (1+w_\infty^{-1/2}\epsilon_1 e^{-\gamma r})+\epsilon_1^2 e^{-2\gamma r}\d r\\
    \leq \ & C_6 \big(\epsilon(1+w_\infty^{-1/2}\epsilon_1)+\epsilon_1^2 \big)\int_t^{t+s} e^{-\gamma r} \d r
    \\
    = \ &C_6 \gamma^{-1}\big(\epsilon(1+w_\infty^{-1/2}\epsilon_1)+\epsilon_1^2 \big) e^{-\gamma t}(1-e^{-\gamma s})\\\leq \ & 2 C_6 E_{\mathrm{max}}(2+\delta_6) e^{1/2}(e^{-\gamma (t+s/2)}-e^{-\gamma (t+s)}),
\end{align*}
where we have used that $1-e^{-x}\leq 2(1-e^{-x/2})$ for all $x\geq 0$ and applied the assumptions $\epsilon_1\leq \min\{\delta_6 w_\infty^{1/2},(\gamma E_{\mathrm{max}})^{1/2}\}$ and $\frac{\epsilon}{\gamma}\leq E_{\mathrm{max}}$. Taking an exponential then gives 
\begin{align*}
    \frac{{\alpha}(t)}{{\alpha}(t+s)}\leq \frac{(e^{2 C_6 E_{\mathrm{max}}(2+\delta_6) e^{1/2}})^{e^{-\gamma (t+s/2)}}}{(e^{2 C_6 E_{\mathrm{max}}(2+\delta_6) e^{1/2}})^{e^{-\gamma (t+s)}}}=\frac{w_+(t+s/2)}{w_+(t+s)},
\end{align*}
and
\begin{equation*}
     \frac{{\alpha}(t+s)}{{\alpha}(t)}\leq \frac{(e^{2 C_6 E_{\mathrm{max}}(2+\delta_6) e^{1/2}})^{-e^{-\gamma (t+s)}}}{(e^{2 C_6 E_{\mathrm{max}}(2+\delta_6) e^{1/2}})^{-e^{-\gamma (t+s/2)}}}= \frac{w_-(t+s)}{w_-(t+s/2)}.   \qedhere
\end{equation*}

% Moving on to the condition on $w_-$, we write
% \begin{align*}
%     \log({\alpha}(t+s))-\log({\alpha}(t))= &\int_t^{t+s} \log({\alpha}(r))^\prime \d r  \\
%     \leq &2 \tilde{C}_1 \gamma^{-1}(\epsilon+\epsilon_1^2 ) e^{\gamma \delta/2}(e^{- \gamma (t+s/2)}-e^{- \gamma (t+s)}),
% \end{align*}
\end{proof}
We finally establish a bound %on the supremum over $s\in [0,t]$ of the quantity
for the quantity
\[q(t,s,\delta)=\tfrac{1}{w_-(t+\delta/2)-w_-(s)}+\tfrac{1}{w_+(s)-w_+(t+\delta/2)},\]
which is used in the proof of \Cref{lem:duhamel}.

\begin{lemma}\label{lem:weight}
Assuming \hyperlink{hyp:initial}{S1} and \hyperlink{hyp:wmin}{S3}, there exists a constant $C_{15}>0$ such that
    \[\sup_{s\in[0,t]} e^{-\gamma  s}|q(t,s,\delta)|\leq C_{15} \frac{e^{\gamma{\delta}/2}}{\gamma{\delta}}, \]
   for all $t,\gamma,\delta>0$.\end{lemma}
\begin{proof}
We first remark that
\[\sup_{s\in[0,t]}\frac{e^{-\gamma  s} }{w_{+}(s)-w_{+}(t+\delta/2)}=\frac{w^2}{w_\infty}\sup_{x\in [0,\gamma   t]}\frac{1}{B(x)},\]
where
\[B(x)=e^x(\tfrac{w}{w_\infty})^{e^{-x}}-e^x(\tfrac{w}{w_\infty})^{e^{-\gamma (t+\delta/2)}}.\]
We now claim that $B$ is decreasing on $[0,\gamma   t]$, so that its infimum is attained at $\gamma  t$. To see this, we compute the derivative 
% \[D(x)=\tfrac{\d }{\d x}e^x\Big((\tfrac{a}{w_\infty})^{e^{-x}}-(\tfrac{a}{w_\infty})^{e^{-\gamma (t+\delta/2)}}\Big)\] 
\begin{align*}
    B^\prime(x)=e^x\Big((1+\log(\tfrac{w_\infty}{w})e^{-x})(\tfrac{w}{w_\infty})^{e^{-x}}-(\tfrac{w}{w_\infty})^{e^{-\gamma (t+\delta/2)}}\Big)
\end{align*}
and use $1+\log(\tfrac{w_\infty}{w})e^{-x}=1+\log((\tfrac{w_\infty}{w})^{e^{-x}})\leq (\tfrac{w_\infty}{w})^{e^{-x}}$ to find
\begin{align*}
    B^\prime(x)\leq e^x\Big(1-(\tfrac{w}{w_\infty})^{e^{-\gamma (t+\delta/2)}}\Big)<0,
\end{align*}
since $\tfrac{w}{w_\infty}>1$. 

Using our claim, we may now estimate
\begin{align*}
    w_+(t)-w_+(t+\delta/2)\geq \tfrac{\delta}{2} |w_+^\prime (t+\delta/2)|
    =\tfrac{\delta}{2} \gamma  \log(\tfrac{w_{\mathrm{min}}}{w}) (\tfrac{w_\infty}{w})^{1-e^{-\gamma  (t+\delta/2)}}e^{-\gamma (t+\delta/2)} ,
\end{align*}
which yields
\begin{align*}
    \frac{e^{-\gamma  t} }{w_+(t)-w_+(t+\delta/2)}\leq2\frac{(\tfrac{w_{\mathrm{min}}}{w})}{\log(\tfrac{w_{\mathrm{min}}}{w})}\frac{ e^{\gamma \delta/2} }{\gamma   \delta }.
\end{align*}
A similar bound for the remaining term involving $w_-$ can be established analogously, completing the proof.
\end{proof}
\bibliographystyle{plain}
\bibliography{references}
\end{document}